\documentclass[12pt,reqno]{amsart}
\usepackage{amssymb}
\usepackage{graphicx}
\usepackage[all]{xy}
\DeclareFontFamily{U}{mathb}{\hyphenchar\font45}
\DeclareFontShape{U}{mathb}{m}{n}{
      <5> <6> <7> <8> <9> <10> gen * mathb
      <10.95> mathb10 <12> <14.4> <17.28> <20.74> <24.88> mathb12
      }{}
\DeclareSymbolFont{mathb}{U}{mathb}{m}{n}
\DeclareMathSymbol{\righttoleftarrow}{3}{mathb}{"FD}

\theoremstyle{plain}
\newtheorem{prop}{Proposition}[section]
\newtheorem{theo}[prop]{Theorem}
\newtheorem{coro}[prop]{Corollary}
\newtheorem{lemm}[prop]{Lemma}
\theoremstyle{remark}
\newtheorem{rema}[prop]{Remark}
\newtheorem{assu}{Assumption}
\theoremstyle{definition}
\newtheorem{defi}[prop]{Definition}

\newtheorem{exam}[prop]{Example}
\numberwithin{equation}{section}

\newcommand{\A}{{\mathbb A}}
\newcommand{\C}{{\mathbb C}}
\newcommand{\PP}{{\mathbb P}}
\newcommand{\bP}{{\mathbb P}}

\newcommand{\G}{{\mathbb G}}
\newcommand{\N}{{\mathbb N}}
\newcommand{\R}{{\mathbb R}}
\newcommand{\Z}{{\mathbb Z}}
\newcommand{\cA}{{\mathcal A}}
\newcommand{\cB}{{\mathcal B}}

\newcommand{\cO}{{\mathcal O}}
\newcommand{\cI}{{\mathcal I}}

\newcommand{\cK}{{\mathcal K}}

\newcommand{\cX}{{\mathcal X}}
\newcommand{\cY}{{\mathcal Y}}

\newcommand{\cU}{{\mathcal U}}

\newcommand{\rH}{{\mathrm H}}
\newcommand{\rK}{{\mathrm K}}
\newcommand{\Wedge}{{\textstyle\bigwedge}}

\newcommand{\eqto}{\stackrel{\lower1.5pt\hbox{$\scriptstyle\sim\,$}}\to}
\newcommand{\eqdashto}{\stackrel{\lower1.5pt\hbox{$\scriptstyle\sim\,$}}\dashrightarrow}
\newcommand{\actsfromleft}{\mathrel{\reflectbox{$\righttoleftarrow$}}}
\newcommand{\actsfromright}{\righttoleftarrow}

\DeclareMathOperator{\Pic}{Pic}
\DeclareMathOperator{\Spec}{Spec}

\DeclareMathOperator{\Hom}{Hom}

\DeclareMathOperator{\Burn}{Burn}

\DeclareMathOperator{\Ind}{Ind}

\DeclareMathOperator{\ord}{ord}
\DeclareMathOperator{\codim}{codim}

\newcommand{\ocB}{{\overline\cB}}
\newcommand{\oBurn}{{\overline\Burn}}

\begin{document}
\title{Equivariant birational types and Burnside volume}

\author{Andrew Kresch}
\address{
  Institut f\"ur Mathematik,
  Universit\"at Z\"urich,
  Winterthurerstrasse 190,
  CH-8057 Z\"urich, Switzerland
}
\email{andrew.kresch@math.uzh.ch}
\author{Yuri Tschinkel}
\address{
  Courant Institute,
  251 Mercer Street,
  New York, NY 10012, USA
}

\email{tschinkel@cims.nyu.edu}

\address{Simons Foundation\\
160 Fifth Avenue\\
New York, NY 10010\\
USA}

\date{July 24, 2020}

\begin{abstract}
We introduce equivariant Burnside groups, 
new invariants in equivariant birational geometry, generalizing 
birational symbols groups for actions of finite abelian groups, due to Kontsevich, Pestun, and the second author, and 
study their properties. We establish a specialization map
for the equivariant birational type of a smooth algebraic variety with an action of a
finite group.
\end{abstract}

\maketitle

\section{Introduction}
\label{sec.intro}
Let $G$ be a finite group, acting on algebraic varieties $X$ and $X'$, defined over a field $k$.   
A classical problem in higher-dimensional algebraic geometry is to determine whether or not there exists a $G$-equivariant birational morphism between $X$ and $X'$. 
This is a formidable challenge already when $G$ is the trivial group and $X'=\bP^n$:
the \emph{rationality problem} has occupied generations of mathematicians and inspired the introduction
of many important ideas and techniques. Recently, there has been a resurgence of interest and activity in rationality questions,
with the introduction of specialization techniques in
\cite{voisin}, \cite{CP}, \cite{NS}, \cite{kontsevichtschinkel},
and their applications to various classes of varieties, e.g.,
in \cite{totaro}, \cite{HKT},  \cite{HPT}, \cite{HT}, \cite{sch1}, \cite{sch2},  \cite{NO}. 

There is also an extensive literature on $G$-equivariant birationality. 
In dimension two, the problem reduces to the study of the $G$-module structure of the Picard lattice $\Pic(X)$,
coupled with the detailed study of the geometry of Del Pezzo surfaces
(see, e.g., \cite{manin}, \cite{DI}, \cite{blanc-thesis}).
The situation is much less clear in higher dimensions, where
the main tools come from the Minimal Model Program and birational rigidity.  

In this paper, we introduce and study a new invariant in $G$-equivariant
birational geometry. 
Informally, it is defined as follows.
Let $X\actsfromright G$ be a smooth projective variety of dimension $n$,
together with a generically free action of $G$.
Then:
\begin{itemize}
\item After a sequence of blow-ups in smooth $G$-invariant subvarieties,
we can assume that the $G$-action on $X$ has only \emph{abelian} stabilizers. 
\item  To each (orbit of a) subvariety with \emph{generic}
abelian stabilizer $H\subseteq G$ we attach a symbol 
\begin{equation}
\label{eqn:symbol}
(H,N_G(H)/H \actsfromleft K,\beta), 
\end{equation}
where $K$ is the function field of the stratum
with action of the normalizer $N_G(H)$ of $H$ in $G$,
such that the restriction to $H$ is the trivial action,
and $\beta$ encodes the generic representation of $H$ on the normal bundle.
\item The invariant $[X \actsfromright G]$ of $X\actsfromright G$ is the sum of all such symbols. 
\end{itemize}
The invariant takes values in the {\em equivariant Burnside group}, 
\[
\Burn_n(G)=\Burn_{n,k}(G),
\]
a quotient of the $\Z$-span of symbols as above by explicit relations, imposed so that
\[
[X \actsfromright G]=[X' \actsfromright G]
\] 
whenever $X$ and $X'$ are $G$-equivariantly birational.
In the definition of
$[X \actsfromright G]$ (Definition \ref{defn.classXG}),
we rely on
\begin{itemize}
\item the divisorialification algorithm of Bergh and Bergh-Rydh, applied to the quotient stack $[X/G]$,
to obtain abelian stabilizers,
\item the $G$-equivariant form of weak factorization.
\end{itemize}
Throughout, we make the assumption that $k$ has characteristic zero,
and we make an additional, simplifying assumption concerning roots of unity in $k$.
(These are stated in Section \ref{sec.prelim-a}.)

When $G$ is trivial, $\Burn_n(G)$ is the Burnside group of fields
\[
\Burn_n=\Burn_{n,k}
\]
defined in \cite{kontsevichtschinkel}.
When $G$ is abelian, $\Burn_n(G)$ admits a surjective homomorphism
to the group $\cB_n(G,k)$ introduced in \cite{kontsevichpestuntschinkel}; the comparison is 
explained in Section \ref{sect:comp}. 
For general $G$, we construct a homomorphism
\[ \Burn_n(G)\to \oBurn_n, \]
to the Burnside group of orbifolds of \cite{Bbar}.
When $X$ is a smooth projective variety with generically free $G$-action,
this sends $[X\actsfromright G]$ to the
class in $\oBurn_n$ of the quotient stack $[X/G]$.

One of our main constructions is a
homomorphism of abelian groups, the
\emph{equivariant Burnside volume} (Definition \ref{defn.equivBurnsidevol})
\[ \rho^G_\pi\colon \Burn_{n,K}(G)\to \Burn_{n,k}(G), \]
when $K$ is the function field of a complete DVR $\mathfrak{o}$ with residue field $k$ and uniformizer $\pi$.
This is an equivariant version of the map of
\cite[\S 5]{kontsevichtschinkel}.
As a consequence we obtain the equivariant specialization of birationality
(Corollary \ref{cor.ebvXG}):
\emph{If
\[ \cX\to \Spec(\mathfrak{o})\qquad\text{and}\qquad \cX' \to \Spec(\mathfrak{o}) \]
are smooth and projective, with generically free actions of a finite group $G$
and $G$-equivariantly birational generic fibers, then the special fibers are
$G$-equivariantly birational.}
 

\medskip
\noindent
\textbf{Acknowledgments:}
We are very grateful to Maxim Kontsevich for essential contributions to
the main constructions and inspiring discussions on these and related questions, and 
to Brendan Hassett for his interest and comments. 
The first author was partially supported by the
Swiss National Science Foundation. The second author was partially supported by NSF grant 2000099.

\section{Preliminaries: algebra}
\label{sec.prelim-a}
In this section we outline an algebraic condition, in terms
of group cohomology, and some of its consequences, that apply to the
function fields in a symbol \eqref{eqn:symbol}.

Let $G$ be a finite group; write $N_G(H)$ for the normalizer of a subgroup $H\subseteq G$. 
Fix a choice of representatives
\[ H_1,\dots,H_r \]
of the conjugacy classes of \emph{abelian} subgroups of $G$ and 
let  $e$ be the least common multiple of the exponents of
$H_1$, $\dots$, $H_r$. Throughout, we work over a field
$k$ of characteristic zero;  for convenience, we assume that $k$ 
contains primitive $e$-th roots of
unity. The character group of an abelian subgroup $H\subseteq G$ will be denoted by 
$$
H^\vee:=\Hom(H,k^\times).
$$

We will use standard facts about Galois algebras $K/K_0$ for a
finite group $\Gamma$, where $K_0$ is a field.
For instance (see, e.g., \cite[\S 4.3]{JLY}):
\begin{itemize}
\item Writing $K\cong K^1\times\dots\times K^\ell$ as product of fields,
for the subgroup $\Gamma^1$ that sends the factor $K^1$ to itself we have
$K$ equivariantly isomorphic to
$\Ind_{\Gamma^1}^\Gamma(K^1)$.
\item (Hilbert's Theorem 90) $H^1(\Gamma,K^\times)=0$.
\end{itemize}

\begin{assu}
\label{assu.weak}
If $H\subseteq G$ is an abelian subgroup, $K_0$ a field containing $k$, and $K/K_0$
a Galois algebra for the group
$N_G(H)/H$, then the composite homomorphism
\begin{equation}
\label{eqn.weak}
\rH^1(N_G(H),K^\times)\to \rH^1(H,K^\times)^{N_G(H)/H}\to H^\vee
\end{equation}
is surjective,
where the first map comes from the Hochschild-Serre spectral sequence
and the second is obtained by writing
\[
K\cong K^1\times\dots\times K^\ell,
\] 
where each $K^i$ is a field, and projecting to
a factor
\[ \rH^1(H,(K^i)^\times)=\Hom(H,(K^i)^\times)=H^\vee. \]
\end{assu}

\begin{rema}
\label{rem.surjequiviso}
That $K_0$ is a field in Assumption \ref{assu.weak} implies:
\begin{itemize}
\item[(i)] The choice of a different factor $K^{i'}$ in
Assumption \ref{assu.weak} modifies the
composite homomorphism \eqref{eqn.weak} by
the automorphism of $H^\vee$, induced by conjugation of $H$ by
a suitable element of $N_G(H)$.
\item[(ii)] Assumption \ref{assu.weak} is therefore
independent of the choice of factor $K^i$.
\item[(iii)] The composite homomorphism \eqref{eqn.weak} is always injective,
since by the Hochschild-Serre spectral sequence,
the kernel is identified with $\rH^1(N_G(H)/H,K^\times)$,
which is zero by Hilbert's Theorem 90.
Thus, an equivalent form of
Assumption \ref{assu.weak} is that the composite homomorphism \eqref{eqn.weak}
is an isomorphism.
\end{itemize}
\end{rema}

\begin{prop}
\label{prop.weak}
Let $H$ be an abelian normal subgroup of a finite group $N$,
$K_0$ a field containing primitive $e$-th roots of unity, where $e$ is the
exponent of $H$, and $K/K_0$ a Galois algebra for the group $N/H$.
Write
\[ K\cong K^1\times\dots\times K^\ell, \]
where $K^1$, $\dots$, $K^\ell$ are fields,
choose some $i\in\{1,\dots,\ell\}$,
and suppose that the composite homomorphism
\[ \rH^1(N,K^\times)\to \rH^1(H,K^\times)^{N/H}\to \rH^1(H,(K^i)^\times) \]
is an isomorphism.
Then, for every $m\in \N$
there is a unique map of non-abelian cohomology sets
\begin{equation}
\label{eqn.nonabeliancoh}
\rH^1(N,GL_m(K))\to \rH^1(H,GL_m(K_0))
\end{equation}
that is compatible with extension of scalars
\[ \rH^1(H,GL_m(K_0))\to \rH^1(H,GL_m(K^i)) \]
and restriction
\[ \rH^1(N,GL_m(K))\to \rH^1(H,GL_m(K^i)), \]
and this map \eqref{eqn.nonabeliancoh} is bijective.
\end{prop}

\begin{proof}
The extension of scalars map is injective (by standard representation theory).
It suffices to exhibit a bijective map \eqref{eqn.nonabeliancoh} that is compatible
with the maps to $\rH^1(H,GL_m(K^i))$.

Let $\zeta\in K_0$ be  a primitive $e$-th root of unity.
We write, according to the structure theorem of finite abelian groups,
\[ H\cong \Z/n_1\Z\times\dots\times \Z/n_r\Z, \quad n_1, \dots, n_r\ge 2, \quad n_i\mid n_{i+1}
\text{  for all } i; \]
with
$n_r=e$.
Sending the $i$-th generator of $H$ to $\zeta^{e/n_i}$ and all other
generators to $1$, we have an element of
$\rH^1(H,K_0^\times)$, which by hypothesis comes from a
$1$-cocycle $(u_{i,g})_{g\in N}$ with values in $K^\times$.
Furthermore, $(u_{i,g}^{n_i})_{g\in N}$ is a
$1$-coboundary, i.e., for some $v_i\in K^\times$ we have
\[ u_{i,g}^{n_i}={}^gv_i/v_i,\qquad\text{for all } g\in N. \]

The data of $(u_{i,g}^{n_i})_{i,g}$ and $(v_i)_i$ give us a way to assign,
functorially, a Galois algebra for the group $H$ over an
\'etale $K_0$-algebra $L_0$ to every Galois algebra $L/L_0$ for the group $N$
with $N$-equivariant $K_0$-algebra homomorphism $K\to L$.
Specifically, given $L/L_0$ and
$\iota\colon K\to L$, for every $i$ we apply Hilbert's Theorem 90 to obtain
$w_i\in L^\times$ such that 
\[ \iota(u_{i,g})={}^gw_i/w_i,\qquad \text{for all } g\in N.
\]
For every $i$, then, $\iota(v_i)w_i^{-n_i}$ is Galois-invariant, i.e., lies in $L_0$,
and is unique up to multiplication by an element of $(L_0^\times)^{n_i}$.
This observation lets us assign, functorially, the
Galois algebra
\[ L_0[t_1,\dots,t_r]/(t_1^{n_1}-\iota(v_i)w_1^{-n_1},\dots,
t_r^{n_r}-\iota(v_r)w_r^{-n_r}) \]
for the group $H$ over $L_0$.
The functorial association is fully faithful.
Essential surjectivity follows easily from the fact that any two
Galois algebras $L_0'/L_0$ and $L_0''/L_0$ for the group $H$ become
equivariantly isomorphic after passage to a suitable \'etale algebra
over $L_0$.

The previous paragraph gives us an equivalence of categories.
This in turn induces an equivalence of categories of the
functorial associations of a free $L_0$-module of rank $m$ to every
Galois algebra $L/L_0$ for the group $N$ with $N$-equivariant $K_0$-algebra
homomorphism $K\to L$,
respectively, to Galois algebras $L'_0/L_0$ for the group $H$.
The isomorphism classes of objects are given by
$\rH^1(N,GL_m(K))$, respectively, $\rH^1(H,GL_m(K_0))$, so the
equivalence of categories gives us a bijective map \eqref{eqn.nonabeliancoh}.
The compatibility with the maps to $\rH^1(H,GL_m(K^i))$ is immediate from
the construction and the fact that $m$-dimensional representations of
$H$ over $K_0$, whose equivalence classes are parametrized by
$\rH^1(H,GL_m(K_0))$, decompose as direct sums of
$1$-dimensional representations.
\end{proof}

\noindent
\textbf{Action construction (A):}
Let
\[ H\subseteq G, \qquad K_0, \qquad N_G(H)/H\actsfromleft K, \] 
satisfy Assumption \ref{assu.weak}, let
\[
a_1,\dots,a_c\in A:=H^\vee
\]
be characters, and define
\[ \overline{H}:=\bigcap_{i=1}^c \ker(a_i), \]
with character group
\[ \overline{A}:=A/\langle a_1,\dots,a_c\rangle. \]
We take
\[ u_1,\dots,u_c \]
to be $1$-cocycles, $u_i=(u_{i,g})_{g\in N_G(H)}$,
with values in $K^\times$, corresponding by the isomorphism that we have,
thanks to Remark \ref{rem.surjequiviso}(iii),
to $a_1$, $\dots$, $a_c$; these determine an action
\[ N_G(H)\actsfromleft K(t_1,\dots,t_c). \]
Writing, as in Assumption \ref{assu.weak},
\[ K\cong K^1\times\dots\times K^\ell, \]
we define $K'$ to be the product of all $K^i$ such that the action
\[ H\actsfromleft K^i(t_1,\dots,t_c) \]
restricts to the \emph{trivial} action of $\overline{H}$.
Then the action of
\[ \overline{H}':=N_{N_G(H)}(\overline{H}) \]
on $K'(t_1,\dots,t_c)$ restricts to the trivial action of $\overline{H}$,
thus we have
\[ \overline{H}'/\overline{H}\actsfromleft K'(t_1,\dots,t_c). \]
We define
\[ \overline{K}:=\Ind_{\overline{H}'}^{N_G(\overline{H})}K'(t_1,\dots,t_c). \]

\section{Preliminaries: geometry}
\label{sec.prelim-g}
Let $G$ be a finite group and $X$ a reduced quasiprojective
scheme over $k$, equipped with a generically free action of $G$. 
By convention, we suppose that the generic point of any component has
dense orbit in $X$.
Then we will say that $X$ is a quasiprojective variety with
generically free $G$-action.
We suppose that $X$ is smooth and remind the reader that this can be
achieved with equivariant resolution of singularities.
(Resolution of singularities is available in a functorial form,
see, e.g., \cite{villamayor}, \cite{bierstonemilman}.)
We will also make use of
functorial weak factorization, established in \cite{abramovichtemkin}.

By a \emph{smooth blow-up} of $X$ we mean the blow-up of $X$ along
a smooth subscheme.
We say that a smooth blow-up is \emph{equivariant} when it is the
blow-up of $X$ along a smooth $G$-invariant subscheme.

\begin{rema}
\label{rem.geom}
Geometry naturally leads to situations where
Assumption \ref{assu.weak} is relevant.
Suppose that $G$ acts on $X$ with abelian stabilizer groups.
Let $V\subset X$ be a reduced subscheme, such that if we express the
decomposition into irreducible components as
\[ V=V^1\cup\dots\cup V^q, \]
then $G$ acts transitively on $\{V^1,\dots,V^q\}$.
The generic stabilizer groups of the $V^i$ comprise a single
conjugacy class of abelian subgroups of $G$, say, the one represented by
$H\in \{H_1,\dots,H_r\}$.
We write
\begin{align*}
\{Y^1,\dots,Y^\ell\}&:=\{V^i\,|\,\text{the stabilizer of the
generic point of $V^i$ is $H$}\}, \\
Y&:=Y^1\cup\dots\cup Y^\ell,
\end{align*}
and note that $N_G(H)$ acts transitively on $\{Y^1,\dots,Y^\ell\}$.
We choose a trivialization of the fiber of the normal bundle
\[ \mathcal{N}_{Y/X} \]
at the generic point of each $Y^i$.
Letting $m$ denote the codimension of $Y$,
the $N_G(H)$-action gives rise to a $1$-cocycle, representing a class in
\[ \rH^1\big(N_G(H),GL_m(k(Y))\big), \]
where $k(Y)$ denotes $k(Y^1)\times\dots\times k(Y^\ell)$.
We have a Galois algebra $k(Y)/k(Z)$ for the group $N_G(H)/H$,
where $Z$ denotes the quotient variety 
$$
Y/N_G(H)\cong V/G,
$$ 
and when $m=1$ (respectively $m>1$)
the class may be studied via the map in Assumption \ref{assu.weak}
(respectively in Proposition \ref{prop.weak}).
\end{rema}

For $H\in \{H_1,\dots,H_r\}$, the $H$-fixed locus $X^H$ is smooth
and stable under the action of $N_G(H)$.
We consider an $N_G(H)$-orbit of components of $X^H$ where the
stabilizer at the generic points of the components is equal to $H$.
Denoting one such by $Y$, we have a generically free action of $N_G(H)/H$
on $Y$ with an irreducible quotient variety $Z$.

\begin{assu}
\label{assu.strong}
The $G$-action on $X$ satisfies: 
\begin{itemize}
\item All stabilizers of the action of $G$ on $X$ are abelian.
\item For every $H\in\{H_1,\dots,H_r\}$, with character group $H^\vee$,
and every $Y$ as above,
the composite homomorphism
\[
\Pic^G(X)
\to \rH^1(N_G(H),k(Y)^\times) 
\to
\rH^1(H,k(Y)^\times)^{N_G(H)/H}\to H^\vee
\]
is surjective.
\end{itemize}
\end{assu}

\begin{rema}
\label{rem.strong}
We make the following observations:
\begin{itemize}
\item[(i)] Assumption \ref{assu.strong} implies
Assumption \ref{assu.weak}, for all abelian $H\subseteq G$ and
Galois algebras as above $k(Y)/k(Z)$ for the group $N_G(H)/H$.
\item[(ii)] Assumption \ref{assu.strong} implies that there exists a
finite collection of $G$-linearized line bundles on $X$ such that the
images of the classes in $\Pic^G(X)$ generate $H^\vee$, for every $H$ and $Y$
as above.
\item[(iii)] Assumption \ref{assu.strong} is equivalent to the existence of
a finite collection of $G$-linearized line bundles on $X$ such that
the associated morphism of stacks
\[ [X/G]\to B\G_m\times\dots\times B\G_m \]
(where the number of factors $B\G_m$ is the number of linearized line bundles)
is representable.
\end{itemize}
In (iii), $[X/G]$ is the quotient stack associated with the
$G$-action on $X$ and $B\G_m$ is the classifying stack of $\G_m$, so that
each $G$-linearized line bundle on $X$ determines
a morphism $[X/G]\to B\G_m$.
For a given finite collection of $G$-linearized line bundles on $X$
the equivalence of the conditions stated in (ii) and (iii) is given in
\cite[Rem.\ 7.14]{bergh}.
\end{rema}

\begin{prop}
\label{prop.divisorialification}
Let $X$ be a smooth quasiprojective variety with a generically free action of $G$.
There exists a sequence of equivariant smooth blow-ups
\[ X'=X_n\to\dots\to X_1\to X_0=X, \]
such that $X'$, with its $G$-action, satisfies Assumption \ref{assu.strong}.
\end{prop}

\begin{proof}
This is a consequence of the \emph{divisorialification} algorithm of
Bergh \cite{bergh} and Bergh-Rydh \cite{berghrydh}, applied to the
quotient stack $[X/G]$.
For the convenience of the reader we explain how this works in the
language of equivariant geometry.
Let $D=D_1\cup\dots\cup D_\ell$ be a simple normal crossing divisor on $X$,
where each $D_i$ is $G$-invariant.
The divisorial index is a quantity attached to a point
$x$ (closed or not) of $X$ and the divisor $D$.
Let $H$ denote the stabilizer of $x$
(so $H$ acts trivially on $\overline{\{x\}}$), and let $\mathcal{N}$ denote the
normal bundle $\mathcal{N}_{\overline{\{x\}}/X}|_{\{x\}}$.
Each $D_i$ that contains $x$ determines a one-dimensional linear representation
of $H$.
Denoting the intersection of the kernels of these representations
by $H'$,
the divisorial index at $x$ is the dimension of the
nontrivial part of $\mathcal{N}$, as a representation of $H'$.

We regard $X$ as equipped with a simple normal crossing divisor
$D=D_1\cup\dots\cup D_\ell$ on $X$,
where each $D_i$ is $G$-invariant: initially $D$ may be taken to
be empty, and with
every blow-up we replace each $D_i$ by its proper transform and adjoin
the exceptional divisor as $D_{\ell+1}$.
Let $m$ denote the maximal value of the divisorial index, over all the
points of $X$.
If $m>0$, then by \cite[\S 7--8]{bergh} and \cite[\S 6]{berghrydh}
the points with divisorial index $m$ are the points of a smooth
$G$-invariant closed subscheme $W\subset X$ which has normal crossing
with $D_1$, $\dots$, $D_\ell$, and after blowing up $X$ along $W$,
every point of the blow-up has divisorial index less than $m$.
Iterating, we achieve $m=0$.
\end{proof}

\begin{rema}
\label{rem.divisorialification}
In the proof of Proposition \ref{prop.divisorialification} we may
start with any
simple normal crossing divisor $D=D_1\cup\dots\cup D_{\ell}$ on $X$,
where each $D_i$ is $G$-invariant.
This transforms in the indicated manner with each blow-up to yield,
finally, a simple normal crossing divisor 
$D'=D'_1\cup\dots\cup D'_{\ell'}$ on $X'$.
A variant, which does not change the variety $X'$ that we get at the end
(up to isomorphism), is to write at every step
\[ W=W_1\cup\dots\cup W_t, \]
a disjoint union, where each $W_i$ is a
$G$-orbit of components, and to index the exceptional divisor over
each $W_i$ separately.
Denoting by $\varphi\colon X'\to X$ the composite map,
with this variant
the support of $\varphi^{-1}(D)$ is necessarily of the form
$\bigcup_{i\in S} D'_i$ for suitable $S\subseteq \{1,\dots,\ell'\}$.
\end{rema}

\begin{exam}
\label{exa.divisorialification}
Even if $k$ contains all roots of unity (e.g., is algebraically closed),
Assumption \ref{assu.weak} is nontrivial.
For instance, let $G=\mathfrak D_8$ be the dihedral group of order $8$, generated by
$\rho$ of order $4$ and $\sigma$ of order $2$. For $H=\langle \rho\rangle$
we consider the Galois algebra $\C(z)/\C(z^2)$ for the group $G/H$.
Assumption \ref{assu.weak} is not satisfied, since
$|\rH^1(G,\C(z)^\times)|=2$.
When this arises geometrically, e.g., from $G$ acting on $\A^3_{\C}$ with
\[
\rho\cdot (x,y,z)=(-y,x,z)\qquad \text{and} \qquad \sigma\cdot (x,y,z)=(x,-y,-z),
\]
divisorialification in the form of Proposition \ref{prop.divisorialification} leads to a situation where
Assumption \ref{assu.strong} is satisfied,
hence as well Assumption \ref{assu.weak}.
\end{exam}

\begin{prop}
\label{prop.equivweakfact}
Let $X$ and $X'$ be smooth projective varieties,
each equipped with a generically free $G$-action,
satisfying Assumption \ref{assu.strong}.
Given a $G$-equivariant birational map
\[
\varphi\colon X'\dashrightarrow X,
\]
restricting to an isomorphism over an open $U\subset X$,
there exists a weak factorization of $\varphi$, where each map is,
or is inverse to, a equivariant smooth blow-up along a center disjoint
from $U$,
and the intermediate varieties in the weak factorization also
satisfy Assumption \ref{assu.strong}.
\end{prop}

\begin{proof}
It suffices, by equivariant resolution of singularities, to treat the case
that $\varphi$ is a morphism.
Then we apply the functorial weak factorization of \cite{abramovichtemkin} and
notice that, since the intermediate varieties admit a
$G$-equivariant morphism to $X$, by Remark \ref{rem.strong}(iii)
they all satisfy Assumption \ref{assu.strong}.
\end{proof}

\begin{rema}
\label{rem.equivweakfact}
The proof of Proposition \ref{prop.equivweakfact} starts by
applying equivariant resolution of singularities to the
closure in $X'\times X$ of the graph of the
restriction of $\varphi$ to $U$, to reduce to the case that
$\varphi$ is a morphism.
In the presence of boundary divisors $D\subset X$, $D=D_1\cup\dots\cup D_\ell$,
and $D'\subset X'$, $D'=D'_1\cup\dots\cup D'_{\ell'}$, simple normal crossing divisors
with respective complement $U$, where
each $D_i$ is $G$-invariant,
equivariant resolution of singularities with boundary divisor lets us
reduce to the case that $\varphi$ is a morphism and
$\varphi^{-1}(D)$ is a simple normal crossing divisor.
The equivariant weak factorization of \cite{abramovichtemkin} is applicable
to $\varphi\colon X'\to X$ with pair of boundary divisors
$(\varphi^{-1}(D),D)$.
\end{rema}

\begin{rema}
\label{rem.notenoughrootsofunity}
The correct formulation of Assumptions \ref{assu.weak} and
\ref{assu.strong}, when $k$ does not contain enough roots of unity,
is that a factor $K^i$ of $K$ contains primitive $e$-th roots of unity,
where $e$ is the exponent of $H$, and the composite map to
$\rH^1(H,(K^i)^\times)$ is surjective.
(This does not imply that $K_0$ contains primitive $e$-th roots of unity,
as we may see by taking $G$ and $H$ as in Example \ref{exa.divisorialification}
and $G/H\actsfromleft \C$ over $K_0:=\R$.)
\end{rema}

\section{Equivariant Burnside group}
\label{sec.eBg}
In this section we define the equivariant Burnside group and study its first properties.
It receives $G$-equivariant birational invariants of $G$-varieties over a field $k$, 
where $G$ is a finite group.  We maintain the assumption that $k$ 
has characteristic zero and contains primitive
$e$-th roots of unity, where $e$ is the least common multiple of the
exponents of abelian subgroups of $G$. 

It will be convenient for the following to proceed in two steps. 

\begin{defi}
\label{defn.eBg0} 
The {\em symbols group} 
$$
\Burn^0_n(G)
$$ 
is the $\Z$-module with

\medskip
\noindent
\textbf{Generators:}
\[
(H,N_G(H)/H \actsfromleft K,\beta),
\]
where
\begin{itemize}
\item $H\subseteq G$ is an abelian subgroup, 
\item $K/K_0$ is a Galois algebra for the group $N_G(H)/H$,
where $K_0$ is a finitely generated field
of transcendence degree $d\le n$ over $k$, up to isomorphism, satisfying
Assumption \ref{assu.weak},
and 
\item $\beta$ is a faithful $(n-d)$-dimensional linear representation of
$H$ over $k$, with trivial space of invariants, up to equivalence;
by Assumption \ref{assu.weak}, $\beta$
decomposes as a sum of one-dimensional representations, hence we may write
$\beta$ as a sequence of characters, up to order:
\[
\beta=(a_1,\ldots, a_{n-d}), \qquad  a_i\in A:=H^\vee.
\]
\end{itemize}
These are subject to the following \textbf{conjugation relations}:

\medskip
\noindent
\textbf{(C1):} Triples with same subgroup and algebra extension are identified,
\[
\qquad\qquad
(H,N_G(H)/H \actsfromleft K,\beta) = (H,N_G(H)/H \actsfromleft K,\beta'),
\]
when $\beta$ and $\beta'$ are related by conjugation by an element of $N_G(H)$.

\medskip
\noindent
\textbf{(C2):} Triples with subgroups and
algebra extensions related by conjugation by $g\in G$ are identified,
\[
(H,N_G(H)/H \actsfromleft K,\beta) = (H',N_G(H')/H' \actsfromleft K,\beta'),
\qquad H'=gHg^{-1},
\]
when $\beta$ and $\beta'$ are related by conjugation by $g$.
\end{defi}

Notice, if $H'=gHg^{-1}=g'Hg'^{-1}$, then the two symbols, identified with
$(H,N_G(H)/H \actsfromleft K,\beta)$ by applying \textbf{(C2)} to $g$ and $g'$,
will be related by \textbf{(C1)}.
By \textbf{(C2)}, the class of any triple in
$\Burn^0_n(G)$ may be expressed as
$(H,N_G(H)/H\actsfromleft K,\beta)$ with
$H\in \{H_1,\dots,H_r\}$.

\begin{defi}
\label{defn.eBg}
The \emph{equivariant Burnside group}
\[ \Burn_n(G)=\Burn_{n,k}(G) \]
is the $\Z$-module with

\medskip
\noindent
\textbf{Generators:}
\[
(H,N_G(H)/H \actsfromleft K,\beta), 
\]
as above, satisfying conjugation relations \textbf{(C1)}--\textbf{(C2)}
and the following \textbf{blow-up relations} $(n-d\ge 2)$:

\medskip
\noindent
\textbf{(B1):} For all $H$, $K$, and $a_1,a_3\ldots, a_{n-d}\in A$ we have
\[ (H,N_G(H)/H \actsfromleft K, (a_1,-a_1,a_3,\dots,a_{n-d}))=0. \]

\medskip
\noindent
\textbf{(B2):} For all $H$, $K$, and $\beta=(a_1,\ldots, a_{n-d})$, $a_i\in A$, 
\[
(H,N_G(H)/H \actsfromleft K, \beta)=\Theta_1+\Theta_2, 
\]
where
\[
\Theta_1= \begin{cases} 0, & \text{if $a_1=a_2$}, \\ 
(H,N_G(H)/H \actsfromleft K, \beta_1)+(H,N_G(H)/H \actsfromleft K, \beta_2),
& \text{otherwise},
\end{cases}
\]
with 
\[
\beta_1:=(a_1,a_2-a_1,a_3,\ldots, a_{n-d}), \quad \beta_2:=(a_2,a_1-a_2,a_3,\ldots, a_{n-d}),
\]
and 
\[
\hskip -1.9cm
\Theta_2= \begin{cases} 0, & \text{if $a_i\in \langle a_1-a_2\rangle$ for some $i$}, \\
(\overline{H},N_G(\overline{H})/\overline{H} \actsfromleft \overline{K}, \bar{\beta}), & \text{otherwise}, 
\end{cases}
\]
with
\[
\overline{A}:=A/\langle a_1-a_2\rangle, \quad \overline{H}^\vee = \overline{A},
\quad 
\bar{\beta}:=(\bar a_2,\bar a_3,\dots,\bar a_{n-d}), \quad
\bar a_i\in \overline{A},
\]
and $\overline{K}$, with the action, described in Construction \textbf{(A)}
in Section \ref{sec.prelim-a},
applied to the character $a_1-a_2$ (with $c=1$).
\end{defi}

\begin{exam}
\label{exa.D4}
Let $G=\mathfrak D_8$ be as in
Example \ref{exa.divisorialification}, and $H:=\langle \rho^2,\sigma\rangle\subset G$.
In $\Burn_2(G)$ we have
\[ (H,G/H \actsfromleft \C\times \C, (a_1,a_2)), \]
with generators $a_1$, $a_2$ of the character group $A$ of $H$, equal to
\begin{align*}
(H,G/H \actsfromleft \C\times \C, (a_1,a_1+a_2))&+
(H,G/H \actsfromleft \C\times \C, (a_2,a_1+a_2)) \\
&\qquad+(\langle \sigma\rangle,\langle \rho^2,\sigma\rangle/\langle \sigma\rangle
\actsfromleft \C(t),a),
\end{align*}
where $\Theta_1$ from \textbf{(B2)} appears on the first line,
and $\Theta_2$, on the second; in a character group of order $2$ we write
$a$ for the non-identity element.
Here, $\Theta_2$ is obtained from the action of $G$
on $\C(t)\times \C(u)$ given by
\begin{align*}
\rho\cdot t&=-u,& \sigma\cdot t&=-t,\\
\rho\cdot u&=t,& \sigma\cdot u&=u,
\end{align*}
and the nontrivial character in $\overline{A}=A/\langle a_1+a_2\rangle$.
Now $\overline{H}=\langle \rho^2\sigma\rangle$
acts trivially only on the first factor $\C(t)$.
So $\Theta_2$ is
\[ (\langle\rho^2\sigma\rangle,\langle \rho^2,\sigma\rangle/\langle \rho^2\sigma\rangle\actsfromleft \C(t),a), \]
written above in an equivalent form by applying \textbf{(C2)}.
\end{exam}

\begin{defi}
\label{defn.classXG}
Let $X$ be a smooth projective variety over $k$ with generically free
$G$-action, satisfying
Assumption \ref{assu.strong}.
Put
\begin{equation}
\label{eqn.classXG}
[X \actsfromright G]:=
\sum_{i=1}^r \sum_{\substack{\text{$Y\subset X$ with}\\ \text{generic stabilizer $H_i$}}}
(H_i,N_G(H_i)/H_i \actsfromleft k(Y),\beta_Y(X)).
\end{equation}
In the inner sum,
$Y$ is an $N_G(H_i)$-orbit of components
where the generic stabilizer is $H_i$.
We understand $k(Y)$ to be the product of the function fields of the
components, and we let
the generic normal bundle representation along $Y$ be recorded as
$\beta_Y(X)$.
\end{defi}

\begin{exam}
\label{exam:show}

There exist projective equivariant compactifications of 
faithful representations  of abelian groups, of the same dimension, which are not
birationally equivalent.  Indeed, by \cite[Theorem 7.1]{reichstein-y} (see also \cite[Section 1]{serre-cremona}), if $V$ and $W$ are $d$-dimensional faithful 
representations of an abelian group $G$ of rank $r\le d$, and 
$\chi_1,\ldots, \chi_d,$ respectively $\eta_1,\ldots,\eta_d\in G^\vee$, are 
the characters appearing in $V$, respectively $W$, then these are $G$-equivariantly birational if and only if
$$
\chi_1\wedge \ldots\wedge\chi_d = \pm \,\eta_1\wedge \ldots \wedge \eta_d\in \Wedge^d(G^\vee).
$$ 
(This condition is nontrivial only when $r=d$.)
\end{exam}

\begin{rema}
In motivic integration, one has considered an equivariant version $\rK_0^G(\mathrm{Var}_k)$ of the Grothendieck group
of algebraic varieties over a field $k$ (see, e.g., \cite[Lemma 5.1]{loo}). 
Note that in the definition of $\rK_0^G(\mathrm{Var}_k)$, one {\em trivializes} the actions of finite groups on 
$\A^n$, which is a significant coarsening, from the birational perspective, 
as can be seen from Example~\ref{exam:show}.  
\end{rema}

\begin{prop}
\label{prop.eBg}
In $\Burn_n(G)$ we have the following relations:
\begin{itemize}
\item[(i)] If $a_1+\dots+a_j=0$, for some $j\in \{2,\ldots, n-d\}$, then 
\[ (H,N_G(H)/H\actsfromleft K,(a_1,\dots,a_{n-d}))=0. \]
\item[(ii)] For any $2\le j\le n-d$,
\[ (H,N_G(H)/H\actsfromleft K,(a_1,\dots,a_{n-d}))=
\sum_{(I,C_I)} (H_I,N_G(H_I)/H_I\actsfromleft K_I, \beta_I)
\]
where the sum is over pairs $(I,C_I)$ such that
\begin{itemize}
\item $I=\{i_0,\dots,i_c\}\subseteq \{1,\dots,j\}$ is \emph{nonempty},
\item $C_I$ is a \emph{nontrivial} coset of the subgroup
$$
\langle a_{i_1}-a_{i_0},\dots,a_{i_c}-a_{i_0}\rangle\subset A,
$$
\item $I=\{1\le i\le j\,|\,a_i\in C_I\}$,
\item the elements
\[
\bar a_{j+1},\dots,\bar a_{n-d} \in A_I:=A/\langle a_{i_1}-a_{i_0},\dots,a_{i_c}-a_{i_0}\rangle
\]
are \emph{nonzero},
\end{itemize}
and Construction \emph{\textbf{(A)}} in Section \ref{sec.prelim-a} is applied to
the characters 
$$
a_{i_1}-a_{i_0}, \quad \dots, \quad a_{i_c}-a_{i_0}
$$ 
to obtain
$$
(H_I,N_G(H_I)/H_i\actsfromleft K_I,\beta_I),
$$ 
with
$H_I:=\overline{H}$, $K_I:=\overline{K}$, and
\[ 
\beta_I:=(\bar a_{i_0},
\bar a_{i'_1}-\bar a_{i_0}, \dots,
\bar a_{i'_{j-c-1}}-\bar a_{i_0},
\bar a_{j+1}, \dots,
\bar a_{n-d}), \]
where $\{i'_1,\dots,i'_{j-c-1}\}\cup I=\{1,\dots,j\}$.
\end{itemize}
\end{prop}

\begin{proof}
We prove (i) and (ii) by induction on $j$, where the base case
$j=2$ is \textbf{(B1)}, respectively \textbf{(B2)}.
For the inductive step of (i), we set
\[ \beta:=(a_1,a_1+a_2,a_3,\dots,a_{n-d}). \]
If $a_1+\dots+a_{j+1}=0$, then \textbf{(B2)} and the induction hypothesis
yield
\[ (H,N_G(H)/H,(a_1,a_2,\dots,a_{n-d}))=0. \]

We carry out the inductive step of (ii).
Apply the induction hypothesis:
\[ (H,N_G(H)/H\actsfromleft K,(a_1,\dots,a_{n-d}))=
\sum_{(I,C_I)} (H_I,N_G(H_I)/H_I\actsfromleft K_I, \beta_I),
\]
and, for each $(I,C_I)$, apply \textbf{(B2)} to 
weights in $\beta_I$ 
corresponding to
$\bar a_{i_0}$ and $\bar a_{j+1}$:
\[ (H,N_G(H)/H\actsfromleft K,(a_1,\dots,a_{n-d}))=
\sum_{(I,C_I)}\Theta_1+\Theta_2. \]
For instance, the contribution of $\Theta_1$ is
\begin{align}
&\sum_{(I,C_I)}\Theta_1= \nonumber \\
&\,\,\sum_{(I,C_I)}(H_I,N_G(H_I)/H_I\actsfromleft K_I,
\beta'_I) \label{term1a} \\
&\,\,+
\sum_{(I,C_I)}(H_I,N_G(H_I)/H_I\actsfromleft K_I,
\beta''_I), \label{term1b}
\end{align}
where
\begin{align*}
\beta'_I&=
(\bar a_{i_0},\bar a_{i'_1}-\bar a_{i_0},\dots,
\bar a_{i'_{j-c-1}}-\bar a_{i_0},\bar a_{j+1}-\bar a_{i_0},
\bar a_{j+2},\dots,\bar a_{n-d}), \\
\beta''_I&=
(\bar a_{j+1},\bar a_{i'_1}-\bar a_{i_0},\dots,
\bar a_{i'_{j-c-1}}-\bar a_{i_0},\bar a_{i_0}-\bar a_{j+1},
\bar a_{j+2},\dots,\bar a_{n-d}).
\end{align*}

Now there are two cases.
First we treat the case $a_{j+1}\notin\{a_1,\dots,a_j\}$.
Then the induction hypothesis identifies
\begin{equation}
\label{eqn.justjplus1}
(H,N_G(H)/H\actsfromleft K,(a_1-a_{j+1},\dots,a_j-a_{j+1},a_{j+1},
a_{j+2},\dots,a_{n-d}))
\end{equation}
with \eqref{term1b}.
The sum of \eqref{term1a}, \eqref{eqn.justjplus1}, and
\[ \sum_{(I,C_I)} \Theta_2 \]
is equal to the required sum of triples over
$(I,C_I)$, $I\subseteq \{1,\dots,j+1\}$.
In the remaining case $a_{j+1}\in\{a_1,\dots,a_j\}$,
we use \textbf{(B1)} to see that \eqref{term1b} vanishes.
What remains gives the required sum of triples.
\end{proof}

\section{Equivariant birational invariants}
\label{sec.ebi}
Our first task is to show that the class introduced in
Definition \ref{defn.classXG}
is an equivariant birational invariant, i.e., 
$$
[X \actsfromright G] \in \Burn_n(G).
$$
Then we extend the definition to include
the case of quasiprojective $G$-varieties.
Finally, we extend the definitions to remove the dependence on
Assumption \ref{assu.strong}.

\begin{theo}
\label{thm.ebi}
Let $X$ and $X'$ be smooth projective varieties of dimension $n$
over $k$, each with a generically free action of a finite group $G$, satisfying
Assumption \ref{assu.strong}.
If $X$ and $X'$ are $G$-equivariantly birationally equivalent, then 
\[ [X \actsfromright G]=[X' \actsfromright G] \]
in $\Burn_n(G)$.
\end{theo}

\begin{proof}
By Proposition \ref{prop.equivweakfact}, it suffices to treat the case that
$X'$ is obtained from $X$ by an equivariant smooth blow-up.
So, let $W$ be a $G$-invariant smooth subscheme of $X$, of
pure dimension $\dim(W)\le n-2$.
Now we split $[X \actsfromright G]$ into two sums:
\begin{align*}
&\sum_{i=1}^r
\sum_{\substack{Y\not\subset W\\ \text{generic stabilizer $H_i$}}}
(H_i,N_G(H_i)/H_i \actsfromleft k(Y),\beta_Y(X)) \\
&\qquad+\sum_{i=1}^r
\sum_{\substack{Y\subset W\\ \text{generic stabilizer $H_i$}}}
(H_i,N_G(H_i)/H_i \actsfromleft k(Y),\beta_Y(X)).
\end{align*}
Letting $E$ denote the exceptional divisor of the blow-up, we similarly split
$[X' \actsfromright G]$ into two sums:
\begin{align*}
&\sum_{i'=1}^r
\sum_{\substack{Y'\not\subset E\\ \text{generic stabilizer $H_{i'}$}}}
(H_{i'},N_G(H_{i'})/H_{i'} \actsfromleft k(Y'),\beta_{Y'}(X')) \\
&\qquad+\sum_{i'=1}^r
\sum_{\substack{Y'\subset E\\ \text{generic stabilizer $H_{i'}$}}}
(H_{i'},N_G(H_{i'})/H_{i'} \actsfromleft k(Y'),\beta_{Y'}(X')).
\end{align*}
The two first sums are equal.
It remains to verify that the two second sums are equal
in $\Burn_n(G)$.
We consider some $Y'\subset E$ in the second sum for
$[X' \actsfromright G]$.
The $G$-orbit of its image in $X$ determines, by the generic stabilizer,
a conjugacy class representative $H_i$, a union of components
$Z$ with generic stabilizer $H_i$, and $Y$ containing $Z$,
appearing in the first or second sum for $[X \actsfromright G]$.
Then
\[ j:=n-\dim(Y)-\dim(W)+\dim(Z) \]
is positive.
Let $\codim(Y)$ be the codimension of $Y$ in $X$;
then $\codim(Y)=n-\dim(Z)$ if and only if $Y\subset W$.
In this case, the corresponding summand from the second sum for
$[X \actsfromright G]$ is equal in $\Burn_n(G)$ to the corresponding terms
from the second sum for $[X' \actsfromright G]$,
by Proposition \ref{prop.eBg} (ii).
Otherwise, we have $\codim(Y)<n-\dim(Z)$,
and we have the vanishing of the terms of the
second sum for $[X' \actsfromright G]$ by \textbf{(B1)}.
\end{proof}

We proceed to define $[U \actsfromright G]\in \Burn_n(G)$ for a smooth quasiprojective
variety $U$ with a faithful action of $G$, satisfying Assumption \ref{assu.strong}.
Imitating
\eqref{eqn.classXG} in the most naive way, we define
\begin{equation}
\label{eqn.naive}
[U \actsfromright G]^{\mathrm{naive}}:=
\sum_{i=1}^r \sum_{\substack{\text{$V\subset U$ with}\\ \text{generic stabilizer $H_i$}}}
(H_i,N_G(H_i)/H_i \actsfromleft k(V),\beta_V(U)).
\end{equation}
This suffers from an important drawback:
A key property of classes in the non-equivariant Burnside group,
which fails for $[U\actsfromright \{1\}]^{\mathrm{naive}}$,
is that when $U\subset X$ is open,
$[X]-[U]\in \Burn_n$ carries information about the boundary $X\setminus U$,
e.g., it is essentially what we get by application of the specialization map
on Burnside groups \cite{kontsevichtschinkel} when
$X$ fibers over a smooth curve $C$ and $U$ is the
pre-image of $C\setminus \{c\}$ for $c\in C(k)$.
However, the following is an immediate consequence of the definition
of $[U \actsfromright G]^{\mathrm{naive}}$.

\begin{lemm}
\label{lem.easy}
Let $U'\subset U$ be a $G$-invariant open subvariety, with the property
that every $V\subset U$ in \eqref{eqn.naive}
has nontrivial intersection with $U'$.
Then in $\Burn_n(G)$ we have
\[ [U \actsfromright G]^{\mathrm{naive}}
=[U' \actsfromright G]^{\mathrm{naive}}. \]
\end{lemm}

The correct definition of $[U \actsfromright G]$
(Definition \ref{defn.classUG}, below) involves an alternating sum over
boundary components, as in \cite[(1.1)]{Bbar}, and the normal bundles
\[
\pi_I\colon \mathcal{N}_{D_I/X}\to D_I,
\]
in $X$ of intersections of divisors
\[
D_I:=\bigcap_{i\in I} D_i, \qquad I\subseteq \cI:=\{1,\dots,\ell\};
\]
we agree by convention that $D_{\emptyset}=X$.

The class $[U\actsfromright G]^{\mathrm{naive}}$ is at least
an equivariant birational invariant of $U$; 
we recall (see, e.g., \cite{Bbar}), that, by definition, $U$ and $U'$ are
equivariantly birationally equivalent if there exist a
quasiprojective $G$-variety $V$ and
equivariant birational projective morphisms
$V\to U$ and $V\to U'$.

\begin{lemm}
\label{lem.Uebi}
Let $U$ and $U'$ be smooth quasiprojective varieties of dimension
$n$ over $k$, each with a generically free action of $G$, satisfying Assumption \ref{assu.strong}.
If $U$ and $U'$ are equivariantly birationally equivalent, then
$$
[U\actsfromright G]^{\mathrm{naive}}=[U'\actsfromright G]^{\mathrm{naive}}.
$$
\end{lemm}

\begin{proof}
As in the proof of Theorem \ref{thm.ebi}, it suffices to treat the
case that $U'$ is obtained from $U$ by an equivariant smooth blow-up.
The proof of Theorem \ref{thm.ebi} carries over without change
to establish the lemma.
\end{proof}

\begin{defi}
\label{defn.classUG}
Let $X$ be a smooth projective variety over $k$ with a generically free $G$-action satisfying Assumption \ref{assu.strong}. 
For smooth quasiprojective 
\[
U=X\setminus D, \qquad D=\bigcup_{i\in \cI} D_i, \qquad \cI:=\{1,\dots,\ell\},
\]
with a generically free $G$-action, and a simple normal crossing divisor $D$, with a
compatible $G$-action, where each $D_i$ is $G$-invariant, we define
\begin{equation}
\label{eqn.classUG}
[U \actsfromright G]:=
[X \actsfromright G] + \sum_{\emptyset\neq I\subseteq \cI} (-1)^{|I|}
[\mathcal{N}_{D_I/X} \actsfromright G]^{\mathrm{naive}}.
\end{equation}
\end{defi}

\begin{rema}
\label{rem.classUG}
In Definition \ref{defn.classUG}, $G$ is not allowed
to permute the divisors $D_1$, $\dots$, $D_\ell$ nontrivially.
For instance, the compactification of 
$$
U:=\A^1\times(\A^1\setminus \{0\}),
$$ 
with
action of $G:=\Z/2\Z$ by 
$$
(x,y)\mapsto (xy^{-1},y^{-1}),
$$ 
by
$\bP^2$ is not permitted; instead we may work with the compactification by
the blow-up of $\bP^2$ at the point where the line $y=0$ meets the line at
infinity.
\end{rema}

It is possible to obtain an alternative formula by recognizing the
cancellation of many terms in \eqref{eqn.classUG}.
We may work in the symbols group
\[ \Burn^0_n(G). \]
Definition \ref{defn.classXG} gives a well-defined class 
of $\Burn^0_n(G)$ attached to $X\actsfromright G$.
We may analogously view
$[U \actsfromright G]^{\mathrm{naive}}$ as an element of
$\Burn^0_n(G)$ and hence, as well, $[U \actsfromright G]$ by the
formula \eqref{eqn.classUG} (whose \emph{a priori} dependence
on the presentation of $U$ as $X\setminus D$ will be removed in
Proposition \ref{prop.classUG}, below).
Lemma \ref{lem.easy} is valid in $\Burn^0_n(G)$.
The canonical homomorphism
\[ \Burn^0_n(G)\to \Burn_n(G) \]
relates the classes in $\Burn^0_n(G)$ to the ones defined in
$\Burn_n(G)$.

\begin{defi}
\label{defn.punctured}
We adopt the notation of Definition \ref{defn.classUG}.
Recognizing, $I\subseteq \cI$, that
$\mathcal{N}_{D_I/X}$ may be identified with the
direct sum over $i\in I$ of the restrictions
$\mathcal{N}_{D_i/X}|_{D_I}$, we define the
\emph{punctured normal bundles}
\[
\mathcal{N}^\circ_{D_I/X}:=\mathcal{N}_{D_I/X}\setminus
\bigcup_{i\in \cI}
\begin{cases}
\bigoplus_{j\in I\setminus\{i\}} \mathcal{N}_{D_j/X}|_{D_I},&\text{when $i\in I$},\\
\pi_I^{-1}(D_{I\cup \{i\}}),&\text{when $i\notin I$},
\end{cases}
\]
with projections
\[
\pi^\circ_I\colon \mathcal{N}^\circ_{D_I/X}\to
D^\circ_I:=D_I\setminus \bigcup_{i\in \cI\setminus I}D_i.
\]
\end{defi}

\begin{lemm}
\label{lem.classUG}
We adopt the notation of Definitions \ref{defn.classUG} and
\ref{defn.punctured}.
Then in $\Burn^0_n(G)$, and hence in $\Burn_n(G)$, we have
\[
[U \actsfromright G]=
[U \actsfromright G]^{\mathrm{naive}}+
\sum_{\emptyset\ne I\subseteq \cI} (-1)^{|I|} [\mathcal{N}^\circ_{D_I/X} \actsfromright G]^{\mathrm{naive}}.
\]
\end{lemm}

\begin{proof}
There is a group homomorphism to $\Burn^0_n(G)$ from
\begin{equation}
\label{eqn.ZJW}
\bigoplus_{H\in \{H_1,\dots,H_r\}}
\bigoplus_{M\subseteq \cI}
\bigoplus_{\substack{\text{$W\subset D_M$ with}\\
\text{generic stabilizer $H$}\\
\text{and $\{i\in \cI\,|\,W\subset D_i\}=M$}}}
\bigoplus_{\substack{\text{$J\subseteq M$ satisfying}\\
\text{(i)--(ii) below}}} \Z[J,W],
\end{equation}
where $A:=H^\vee$ and
$a_i\in A$ is determined by the divisor $D_i$ for $i\in M$,
so that $A$ is generated by $a_i$ for $i\in M$ and $\beta_W(D_M)$
(characters of the generic normal bundle representation along $W$ in $D_M$),
and $J\subseteq M$ in the final sum is required to satisfy:
\begin{itemize}
\item[(i)] $\{j\in M\,|\,a_j\in \langle a_i\rangle_{i\in M\setminus J}\}=M\setminus J$;
\item[(ii)] no character in $\beta_W(D_M)$ lies in $\langle a_i\rangle_{i\in M\setminus J}$.
\end{itemize}
Letting $H_J\subseteq H$ be the subgroup with
$H_J^\vee=A/\langle a_i\rangle_{i\in M\setminus J}$,
then, $[J,W]$ is mapped to the triple
\[
(H_J, N_G(H_J)/H_J\actsfromleft K_J, \beta_J),
\]
where $N_G(H_J)/H_J\actsfromleft K_J$ arises by application of
Construction \textbf{(A)} in Section \ref{sec.prelim-a}
to $N_G(H)/H\actsfromleft k(W)$ and $(a_i)_{i\in M\setminus J}$, and
\[ \beta_J=\gamma_J\oplus \beta_W(D_M)|_{H_J}, \]
with $\gamma_J$ given by $\bar a_i$ for $i\in J$.

We interpret \eqref{eqn.classUG} as taking
values in the group \eqref{eqn.ZJW} as follows.
Given $I\subseteq \cI$ and $V$ appearing in the
definition of $[\mathcal{N}_{D_I/X}\actsfromright G]^{\mathrm{naive}}$,
we associate $[J,W]$ for $W=V\cap D_I$, where with
$M=\{i\in \cI\,|\,W\subset D_i\}$ there is a unique subset $J$ such that
$I\cup J=M$ and, for the vector bundle
\[ \varphi_{J,M}\colon \bigoplus_{i\in M\setminus J} \mathcal{N}_{D_i/X}|_{D_M}\to D_M, \]
we have $V=\varphi_{J,M}^{-1}(W)$.
We observe that $[\mathcal{N}^\circ_{D_I/X}\actsfromright G]^{\mathrm{naive}}$
contributes just the terms with $J=\emptyset$.

The proposition is thus reduced to the observation that
in the group \eqref{eqn.ZJW} the sum of terms from \eqref{eqn.classUG} with
$J\ne \emptyset$ is zero.
For given $W$ (which determines $M$) and $J\ne \emptyset$,
each $I\subseteq M$ with $I\cup J=M$ appears with coefficient
$(-1)^{|I|}$.
These coefficients sum to zero.
\end{proof}

\begin{prop}
\label{prop.classUG}
Let $U\actsfromright G$ be a smooth quasiprojective variety of dimension $n$ over $k$ with 
a generically free $G$-action.  Then its associated class in $\Burn^0_n(G)$, and hence in 
$\Burn_n(G)$, is independent of the choice of presentation as $U=X\setminus D$, i.e., 
if $X'$ is another smooth projective variety 
with a generically free $G$-action, with an equivariant embedding $U\to X'$ with complement
\[
D'=\bigcup_{i\in \cI'} \, D_{i}',\qquad
\cI':=\{1,\dots,\ell'\},
\] 
a simple normal crossing
divisor, where each $D'_i$ is $G$-invariant, such that $X'$ satisfies
Assumption \ref{assu.strong}, then formula \eqref{eqn.classUG}
agrees with the analogous formula
for $X'$ and $D'$.
\end{prop}

The proof of Proposition \ref{prop.classUG}, along with the
birational invariance of $[U\actsfromright G]$ stated below
as Proposition \ref{prop.classUebi}, will use
Lemmas \ref{lem.Uebi} and \ref{lem.classUG}, along with
a careful cancellation argument much like that used in the proof
of Lemma \ref{lem.classUG}.
Because of the intricate combinatorics, we present first an
example in a simple setting to illustrate the cancellation scheme.

\begin{exam}
\label{exa.classUG}
Let $X$ be a smooth projective variety of dimension $n$ with
a generically free $G$-action, satisfying Assumption \ref{assu.strong}.
Let $U=X\setminus D$, where $D$ is a smooth invariant divisor,
let $d\ge 2$, and let $Z$ be a smooth purely $(n-d)$-dimensional subscheme
of $D$:
\[ Z\subset D\subset X. \]
Let $\widetilde{X}$ be the blow-up of $Z$ in $X$ with
exceptional divisor $E$ and proper transform $\widetilde{D}$ of $D$.
Then an instance of Proposition \ref{prop.classUG} is the following equality:
\begin{align*}
&[X\actsfromright G]-[\mathcal{N}_{D/X}\actsfromright G]^{\mathrm{naive}}= \\
&\,\,\,\,\,\,[\widetilde{X}\actsfromright G]-[\mathcal{N}_{\widetilde{D}/\widetilde{X}}\actsfromright G]^{\mathrm{naive}}
-[\mathcal{N}_{E/\widetilde{X}}\actsfromright G]^{\mathrm{naive}}
+[\mathcal{N}_{\widetilde{D}\cap E/\widetilde{X}}\actsfromright G]^{\mathrm{naive}}.
\end{align*}
As a first step, we apply Lemma \ref{lem.classUG} to both sides.
Then the terms $[U\actsfromright G]^{\mathrm{naive}}$ on each side cancel,
leaving us to verify
\begin{align}
\begin{split}
\label{eqn.exatoverify}
&[\mathcal{N}^\circ_{D/X}\actsfromright G]^{\mathrm{naive}}
-[\mathcal{N}^\circ_{\widetilde{D}/\widetilde{X}}\actsfromright G]^{\mathrm{naive}}= \\
&\qquad\qquad [\mathcal{N}^\circ_{E/\widetilde{X}}\actsfromright G]^{\mathrm{naive}}
-[\mathcal{N}^\circ_{\widetilde{D}\cap E/\widetilde{X}}\actsfromright G]^{\mathrm{naive}}.
\end{split}
\end{align}
The left-hand side of \eqref{eqn.exatoverify} is just the contribution to
$[\mathcal{N}^\circ_{D/X}\actsfromright G]^{\mathrm{naive}}$
from loci $V$ with various generic stabilizer groups that are contained in
the pre-image under $\mathcal{N}^\circ_{D/X}\to D$ of $Z$.
In the spirit of the proof of Lemma \ref{lem.classUG}, these may be
labeled by $W\subset Z$ with some
generic stabilizer $H\in \{H_1,\dots,H_r\}$, where the divisor $D$ determines
an element $a\in A:=H^\vee$, such that
no character of $\beta_W(D)$ lies in $\langle a\rangle$.
The left-hand side of \eqref{eqn.exatoverify} is a sum over
such $W$ of elements of $\Burn_n(G)$
arising by Construction \textbf{(A)} in Section \ref{sec.prelim-a} from
$N_G(H)/H\actsfromleft k(W)$ and $a$, with the
representation obtained by restriction from $\beta_W(D)$.
On the right-hand side of \eqref{eqn.exatoverify}, we have
\begin{equation}
\label{eqn.N0}
\mathcal{N}^\circ_{E/\widetilde{X}}\cong \mathcal{N}_{Z/X}\setminus \mathcal{N}_{Z/D},
\end{equation}
while $\mathcal{N}_{\widetilde{D}\cap E/\widetilde{X}}$
is a direct sum of the line bundles
\begin{equation}
\label{eqn.N1}
\mathcal{N}_{\widetilde{D}\cap E/\widetilde{D}}\cong \cO_{\PP(\mathcal{N}_{Z/D})}(-1)
\end{equation}
and
\begin{equation}
\label{eqn.N2}
\mathcal{N}_{\widetilde{D}\cap E/E}\cong \mathcal{N}_{D/X}|_{\PP(\mathcal{N}_{Z/D})}\otimes
\cO_{\PP(\mathcal{N}_{Z/D})}(1).
\end{equation}
Here, for \eqref{eqn.N1}--\eqref{eqn.N2} we use
$E\cong \PP(\mathcal{N}_{Z/X})$ and
$\widetilde{D}\cap E\cong \PP(\mathcal{N}_{Z/D})$.
So,
\begin{equation}
\label{eqn.N3}
\mathcal{N}^\circ_{\widetilde{D}\cap E/\widetilde{X}}\cong
(\mathcal{N}_{Z/D}\setminus Z) \times_D (\mathcal{N}_{D/X}\setminus D).
\end{equation}
We relate the two punctured normal bundles on the right-hand side of
\eqref{eqn.exatoverify} by
choosing, Zariski locally on $Z$, a $G$-equivariant splitting of the
short exact sequence of vector bundles
\begin{equation}
\begin{split}
\label{eqn.splitting}
\xymatrix{
0 \ar[r] &
\mathcal{N}_{Z/D} \ar[r] &
\mathcal{N}_{Z/X} \ar[r] &
\mathcal{N}_{D/X}|_Z \ar[r] \ar@/_12pt/[l]_s &
0.
}
\end{split}
\end{equation}
This is possible, since
splittings exist Zariski locally and
may be averaged with their translates to yield equivariant splittings,
or in fancier language, since
the invariant local section functor on coherent sheaves is exact
\cite[Lemma 2.3.4]{abramovichvistoli}.
Using \eqref{eqn.N0} and \eqref{eqn.N3},
we obtain from \eqref{eqn.splitting} isomorphisms
\[
\mathcal{N}^\circ_{E/\widetilde{X}}\setminus s(\mathcal{N}_{D/X}|_Z)\cong
\mathcal{N}^\circ_{\widetilde{D}\cap E/\widetilde{X}}.
\]
over $G$-invariant Zariski open subsets of $Z$.
A splitting does not necessarily exist globally
but does exist and is unique upon restricting to the locus
$Z'$ where no character of $\beta_Z(D)$ is equal to $a$,
by Schur's lemma.
This observation, together with the fact that every $W$ as above is
contained in $Z'$, lets us identify
(see Lemma \ref{lem.locallyisomorphic})
the right-hand side of \eqref{eqn.exatoverify} with a sum over
$W\subset Z$ as before
and thereby establish \eqref{eqn.exatoverify}.
\end{exam}

\begin{lemm}
\label{lem.locallyisomorphic}
Let $d\le n$, let $Z$ be a smooth purely $(n-d)$-dimensional quasiprojective $G$-scheme
over $k$, and let
$U$ and $U'$ be quasiprojective purely $n$-dimensional schemes over $k$, with
generically free $G$-actions and $G$-equivariant smooth morphisms
\[ U\to Z\qquad\text{and}\qquad U'\to Z. \]
Suppose that for every $z\in Z$ there exist a $G$-invariant Zariski
neighborhood $Y\subset Z$ of $z$ and a $G$-equivariant open immersion
\[ U'\times_ZY\to U\times_ZY \]
which commutes with the projection maps to $Y$
and satisfies the condition in Lemma \ref{lem.easy}.
Then in $\Burn^0_n(G)$ we have
\[ [U \actsfromright G]^{\mathrm{naive}}=[U' \actsfromright G]^{\mathrm{naive}}. \]
\end{lemm}

\begin{proof}
There exists a finite collection $Y_1$, $\dots$, $Y_m$ of
$G$-invariant Zariski neighborhoods as in the statement of the lemma,
whose union is $Z$.
Over each $Y_j$ we choose an open immersion as in the statement.
In the definition of $[U' \actsfromright G]^{\mathrm{naive}}$, to each
$1\le i\le r$ and $V'\subset U'$ with generic stabilizer $H_i$ we
associate $j\in \{1,\dots,m\}$, taken to be minimal with the
property that the image of $V'$ in $Z$ has nontrivial intersection with
$Y_j$.
Then we use the chosen $G$-equivariant open immersion over $Y_j$
to identify a corresponding $V\subset U$ with
generic stabilizer $H_i$, appearing in the definition of
$[U \actsfromright G]^{\mathrm{naive}}$.
The hypotheses guarantee that every $V\subset U$ appearing in the
definition of $[U \actsfromright G]^{\mathrm{naive}}$ is accounted for.
\end{proof}

\begin{proof}[Proof of Proposition \ref{prop.classUG}]
By equivariant weak factorization (see Remark \ref{rem.equivweakfact})
it suffices to compare the presentations of $U$ as $X\setminus D$ and
$\widetilde{X}\setminus \widetilde{D}$, where
$\widetilde{X}$ is obtained from $X$ by equivariant smooth blow-up, with a
center of blow-up $Z$ that is disjoint from $U$ and has
normal crossing with the divisors $D_i$, and
\[ \widetilde{D}=\widetilde{D}_1\cup\dots\cup \widetilde{D}_\ell\cup E, \]
where $\widetilde{D}_i$ denotes the proper transform of $D_i$
for $i=1$, $\dots$, $\ell$, and $E$, the exceptional divisor.

We introduce the notation
\[ \cI':=\{i\in \cI\,|\,Z\subset D_i\}\qquad\text{and}\qquad
\cI'':=\{i\in \cI\,|\,Z\not\subset D_i\} \]
and, for $I\subseteq \cI$,
\[ I':=\cI'\cap I\qquad\text{and}\qquad I'':=\cI''\cap I. \]
Additionally, we define
\[ Z_I:=D^\circ_{\cI'\cup I}\cap Z. \]

Application of Lemma \ref{lem.classUG} to the
expressions for $[U \actsfromright G]$ from
$U\subset X$ and $U\subset X'$ reduces the proposition to the verification of
\begin{align*}
\sum_{\emptyset\ne I\subseteq \cI}(-1)^{|I|}&
\big([\mathcal{N}^\circ_{D_I/X}\actsfromright G]^{\mathrm{naive}}
-[\mathcal{N}^\circ_{\widetilde{D}_I/\widetilde{X}}\actsfromright G]^{\mathrm{naive}}\big)=\\
&\qquad\sum_{I\subseteq \cI}(-1)^{|I|+1}
[\mathcal{N}^\circ_{\widetilde{D}_I\cap E/\widetilde{X}}\actsfromright G]^{\mathrm{naive}}.
\end{align*}
This is equivalent to
\begin{align}
\begin{split}
\label{eqn.rearrange}
&\sum_{\emptyset\ne I\subseteq \cI}(-1)^{|I|}
\big([\mathcal{N}^\circ_{D_I/X}\actsfromright G]^{\mathrm{naive}}
-[\mathcal{N}^\circ_{\widetilde{D}_I/\widetilde{X}}\actsfromright G]^{\mathrm{naive}}\big)=\\
&\qquad\sum_{\emptyset\ne I\subseteq \cI}(-1)^{|I|}
\big([\mathcal{N}^\circ_{\widetilde{D}_{I''}\cap E/\widetilde{X}}\actsfromright G]^{\mathrm{naive}}
-[\mathcal{N}^\circ_{\widetilde{D}_I\cap E/\widetilde{X}}\actsfromright G]^{\mathrm{naive}}\big).
\end{split}
\end{align}
Indeed, taking $i_0\in \cI'$,
the summand with $I=\{i_0\}$ has first term
$-[\mathcal{N}^\circ_{E/\widetilde{X}}\actsfromright G]^{\mathrm{naive}}$,
while the first terms from the remaining summands cancel, as we see by
pairing the terms indexed by $I$ and $I\cup \{i_0\}$ for
$\emptyset\ne I\subseteq \cI\setminus\{i_0\}$.

For any $I\subseteq \cI$ and $i\in \cI'$ we have
$D_I\setminus D_i\cong \widetilde{D}_I\setminus (\widetilde{D}_i\cup E)$.
It follows easily that the left-hand side of \eqref{eqn.rearrange} is $0$
whenever $\cI'\not\subseteq I$.
When $\cI'\subseteq I$, the
left-hand side of \eqref{eqn.rearrange} is equal to the
contribution to
$[\mathcal{N}^\circ_{D_I/X}\actsfromright G]^{\mathrm{naive}}$
from loci $V$ with various stabilizer groups, contained in
\[ (\pi^\circ_I)^{-1}(Z_I). \]
Any such $V$ is equal to $(\pi^\circ_I)^{-1}(W)$, where $W$ is the
image of $V$ in $Z_I$.
Let $H$ be the generic stabilizer of
(a union of components of) $W$.
Then the contribution, specifically, is gotten by pairing the outcome of
Construction \textbf{(A)} in Section \ref{sec.prelim-a}, applied to
$N_G(H)/H\actsfromleft k(W)$ and
$(a_i)_{i\in I}$, with the characters of
$\beta_W(D^\circ_I)$.
We have
\[ W\subset Z', \]
where
\[ Z'\subset Z_I \]
is defined to be the locus where, for all $i\in \cI'$,
the character $a_i$ in the corresponding character group
does not appear in $\beta_{Z_I}(D^\circ_I)$.

We now turn to an analysis of the right-hand side of \eqref{eqn.rearrange}.
Since $Z$ meets $D_{I''}$ transversally we have
\[ \mathcal{N}_{E/\widetilde{X}}|_{\widetilde D_{I''}\cap E}
\cong \mathcal{N}_{\widetilde D_{I''}\cap E/\widetilde D_{I''}}, \]
hence as well
\[ \mathcal{N}_{\widetilde{D}_{I''}\cap E/\widetilde{X}}\cong
\mathcal{N}_{\widetilde D_{I''}\cap E/\widetilde D_{I''}}
\oplus \Big(\bigoplus_{j\in I''}\mathcal{N}_{D_j/X}|_{\widetilde D_{I''}\cap E}\Big).
\]
The complement of the zero-section in the line bundle
$\mathcal{N}_{\widetilde D_{I''}\cap E/\widetilde D_{I''}}$ may be
identified with the complement of the zero-section in the vector bundle
$\mathcal{N}_{D_I\cap Z/D_{I''}}$.
To obtain $\mathcal{N}^\circ_{\widetilde{D}_{I''}\cap E/\widetilde{X}}$,
we remove the exceptional divisors of the blow-up of
$D_{I''\cup \{i\}}$ along $D_I\cap Z$ for every $i\in \cI'$
as well as everything over
divisors indexed by $\cI''\setminus I$:
\[
\mathcal{N}^\circ_{\widetilde{D}_{I''}\cap E/\widetilde{X}}\cong
\Big(\mathcal{N}_{Z_I/D_{I''}}
\setminus
\bigcup_{i\in \cI'}\mathcal{N}_{Z_I/D_{I''\cup\{i\}}}\Big)\times_{Z_I}
\Big(
\mathop{\raisebox{-.5ex}{\hbox{\huge{$\times$}}}}\limits_{j\in I''}
(\mathcal{N}_{D_j/X}|_{Z_I}\setminus Z_I)
\Big).
\]
Analogously, in the respective cases $\cI'\subseteq I$ and
$\cI'\not\subseteq I$ we have
\[
\mathcal{N}^\circ_{\widetilde{D}_I\cap E/\widetilde{X}}\cong
\Big(\mathcal{N}_{Z_I/D_I}
\setminus Z_I\Big)\times_{Z_I}
\Big(
\mathop{\raisebox{-.5ex}{\hbox{\huge{$\times$}}}}\limits_{j\in I}
(\mathcal{N}_{D_j/X}|_{Z_I}\setminus Z_I)
\Big),
\]
respectively
\[
\mathcal{N}^\circ_{\widetilde{D}_I\cap E/\widetilde{X}}\cong
\Big(\mathcal{N}_{Z_I/D_I}
\setminus \bigcup_{i\in\cI'\setminus I}\mathcal{N}_{Z_I/D_{I\cup\{i\}}}\Big)\times_{Z_I}
\Big(
\mathop{\raisebox{-.5ex}{\hbox{\huge{$\times$}}}}\limits_{j\in I}
(\mathcal{N}_{D_j/X}|_{Z_I}\setminus Z_I)
\Big).
\]

As explained in Example \ref{exa.classUG},
Zariski locally over $Z_I$ there exist equivariant splittings of the
short exact sequence
\begin{equation}
\label{eqn.needsplit}
0\to \mathcal{N}_{Z_I/D_{\cI'\cup I}}\to \mathcal{N}_{Z_I/D_{I''}}\to
\bigoplus_{j\in \cI'}\mathcal{N}_{D_j/X}|_{Z_I}\to 0.
\end{equation}
First we treat the case $\cI'\subseteq I$.
Zariski locally over $Z_I$ we use splittings of \eqref{eqn.needsplit}
to identify
\[
\mathcal{N}_{Z_I/D_I}\oplus \big(\bigoplus_{j\in \cI'}\mathcal{N}_{D_j/X}|_{Z_I}\big)
\]
with $\mathcal{N}_{Z_I/D_{I''}}$.
We have also defined $Z'\subset Z_I$, closed, over which a splitting $s$ of
\eqref{eqn.needsplit} exists and is unique.
So
\begin{align*}
U:=
\Big(\mathcal{N}_{Z_I/D_{I''}}\setminus
\Big(\bigcup_{i\in \cI'}\mathcal{N}_{Z_I/D_{I''\cup\{i\}}}\cup
s&\big(\bigoplus_{i\in \cI'}\mathcal{N}_{D_i/X}|_{Z'} \big)\Big)\Big)\times_{Z_I}\\
&\Big(\mathop{\raisebox{-.5ex}{\hbox{\huge{$\times$}}}}\limits_{j\in I''}
(\mathcal{N}_{D_j/X}|_{Z_I}\setminus Z_I)\Big)
\end{align*}
is a well-defined open subscheme of
$\mathcal{N}^\circ_{\widetilde{D}_{I''}\cap E/\widetilde{X}}$.
To $U$ we have, over $G$-invariant open subschemes of $Z_I$,
equivariant open immersions from
$\mathcal{N}^\circ_{\widetilde{D}_I\cap E/\widetilde{X}}$.
Lemma \ref{lem.locallyisomorphic} is applicable and yields
\[ [U \actsfromright G]^{\mathrm{naive}}=[\mathcal{N}^\circ_{\widetilde{D}_I\cap E/\widetilde{X}} \actsfromright G]^{\mathrm{naive}} \]
in $\Burn^0_n(G)$ (or both $U$ and $\widetilde{D}_I\cap E$ are empty, and the equality holds trivially).
On the other hand, we may write
\[ 
[\mathcal{N}^\circ_{\widetilde{D}_{I''}\cap E/\widetilde{X}} \actsfromright G]^{\mathrm{naive}}-
[U \actsfromright G]^{\mathrm{naive}}
\]
as a sum over $W\subset Z_I$ of elements of $\Burn^0_n(G)$, which is equal to the
contribution to $[\mathcal{N}^\circ_{D_I/X}\actsfromright G]^{\mathrm{naive}}$
from loci contained in $(\pi^\circ_I)^{-1}(Z_I)$.
The equality of the left- and right-hand sides of \eqref{eqn.rearrange} is
thus established for the summands with $\cI'\subseteq I$.

When $\cI'\not\subseteq I$ we take a splitting of \eqref{eqn.needsplit}
over invariant open $Y\subset Z_I$
\[ s\colon \bigoplus_{j\in \cI'}\mathcal{N}_{D_j/X}|_{Y}\to
\mathcal{N}_{Y/D_{I''}} \]
and restrict
to $\bigoplus_{j\in I'} \mathcal{N}_{D_j/X}|_{Y}$ to obtain an
isomorphism
\begin{equation}
\label{eqn.localiso}
\mathcal{N}_{Y/D_I}\oplus \Big(\bigoplus_{j\in I'} \mathcal{N}_{D_j/X}|_Y\Big)
\cong \mathcal{N}_{Y/D_{I''}}.
\end{equation}
This, we claim, induces an isomorphism of
$\mathcal{N}^\circ_{\widetilde{D}_I\cap E/\widetilde{X}}$ with
$\mathcal{N}^\circ_{\widetilde{D}_{I''}\cap E/\widetilde{X}}$
over $Y$.
Indeed, when $i\in I'$, the isomorphism \eqref{eqn.localiso} induces
\[
\mathcal{N}_{Y/D_I}\oplus
s\big(\bigoplus_{j\in I'\setminus \{i\}}\mathcal{N}_{D_j/X}|_Y\big)\cong
\mathcal{N}_{Y/D_{I''\cup\{i\}}},
\]
while for $i\in \cI'\setminus I$ we obtain
\[
\mathcal{N}_{Y/D_{I\cup \{i\}}}\oplus
s\big(\bigoplus_{j\in I'}\mathcal{N}_{D_j/X}|_Y\big)\cong
\mathcal{N}_{Y/D_{I''\cup\{i\}}}
\]
from \eqref{eqn.localiso}.
For the latter, it is crucial that the splitting used
to produce \eqref{eqn.localiso} is the restriction of a
splitting $s$ of \eqref{eqn.needsplit}.
Lemma \ref{lem.locallyisomorphic} is applicable (where the
open immersions that exist locally are isomorphisms), so that
\[
[\mathcal{N}^\circ_{\widetilde{D}_{I''}\cap E/\widetilde{X}}\actsfromright G]^{\mathrm{naive}}
-[\mathcal{N}^\circ_{\widetilde{D}_I\cap E/\widetilde{X}}\actsfromright G]^{\mathrm{naive}}=0
\]
in $\Burn^0_n(G)$, as desired.
\end{proof}

\begin{rema}
\label{rem.specialZ}
It is instructive to examine the proof of Proposition \ref{prop.classUG}
when $Z$ is equal to
the intersection of some of the divisors $D_i$.
Then the left-hand term in \eqref{eqn.needsplit} is always zero.
The equalities
\begin{align*}
[\mathcal{N}^\circ_{D_I/X}\actsfromright G]^{\mathrm{naive}}&=
[\mathcal{N}^\circ_{\widetilde{D}_{I''}\cap E/\widetilde{X}}\actsfromright G]^{\mathrm{naive}}
&&(\cI'\subseteq I)\qquad\text{and}\\
[\mathcal{N}^\circ_{\widetilde{D}_I\cap E/\widetilde{X}}\actsfromright G]^{\mathrm{naive}}&=
[\mathcal{N}^\circ_{\widetilde{D}_{I''}\cap E/\widetilde{X}}\actsfromright G]^{\mathrm{naive}}
&&(\cI'\not\subseteq I)
\end{align*}
can be traced back to canonical, globally defined $G$-equivariant isomorphisms
of the respective pairs of punctured normal bundles.
A general result, from which these isomorphisms can (indirectly) be obtained
is stated next.
\end{rema}

\begin{prop}
\label{prop.classUGrelated}
Let $W$ be a smooth quasiprojective variety with a generically free $G$-action,
$D=D_1\cup\dots\cup D_\ell$ a simple normal crossing divisor where
each $D_i$ is $G$-invariant, and $a_1$, $\dots$, $a_\ell$ positive integers.
Let $U:=W\setminus D$ and
\[ \cK:=i_*\cO_U, \]
where $i\colon U\to W$ denotes the inclusion.
and define the $\cO_W$-subalgebra
\[ \cA:=\{f\,|\,a_1\ord_1(f)+\dots+a_\ell\ord_\ell(f)\ge 0\} \]
of $\cK$,
where $\ord_i(f)$ denotes the order of vanishing of $f\in \cK_z$ along $D_i$
and the condition is imposed at all $z\in Z:=D_1\cap\dots\cap D_\ell$.
Then $V:=\Spec(\cA)$ is smooth with smooth divisor $E\subset V$
defined by the sheaf of ideals
\[ \cI:=\{f\,|\,a_1\ord_1(f)+\dots+a_\ell\ord_\ell(f)>0\}, \]
and we have a canonical $G$-equivariant isomorphism
\[ \mathcal{N}^\circ_{E/V}\cong \mathcal{N}^\circ_{Z/W}. \]
\end{prop}

\begin{prop}
\label{prop.classUebi}
Let $U$ and $U'$ be smooth quasiprojective varieties with
generically free $G$-actions, of the form
$U=X\setminus D$ and $U'=X'\setminus D'$ for
simple normal crossing divisors
\[
D=\bigcup_{i\in \cI}D_i,\qquad \cI:=\{1,\dots,\ell\},
\]
and
\[
D'=\bigcup_{i\in \cI'}D'_i,\qquad \cI':=\{1,\dots,\ell'\},
\]
where $X$ and $X'$ have
compatible $G$-actions, satisfy Assumption \ref{assu.strong},
and each $D_i$ and each $D'_i$ is
$G$-invariant.
If $U$ and $U'$ are equivariantly birationally equivalent, then 
\[ [U \actsfromright G]=[U' \actsfromright G] \]
in $\Burn_n(G)$.
\end{prop}

\begin{proof}
As in the proof of
Proposition \ref{prop.classUG} it suffices to treat the case that
$X'$ is obtained from $X$
by equivariant smooth blow-up, where the center of blow-up $Z$ has
normal crossing with the divisors $D_i$, and the components of
$Z$ form a single $G$-orbit.
The case that $Z$ is disjoint from $U$ has already been treated,
so we may suppose that $Z$ meets $D_I$ transversely, for every $I\subseteq \cI$.
Then we have $\ell'=\ell$, and for every $I\subseteq \cI$,
\[ \mathcal{N}_{D'_I/X'}\cong \mathcal{N}_{D_I/X}|_{D'_I}. \]
By Lemma \ref{lem.Uebi},
\[ [\mathcal{N}_{D_I/X} \actsfromright G]^{\mathrm{naive}}=
[\mathcal{N}_{D'_I/X'} \actsfromright G]^{\mathrm{naive}} \]
in $\Burn_n(G)$.
By combining this with the equality from Theorem \ref{thm.ebi} we get the
desired result.
\end{proof}

Using Proposition \ref{prop.divisorialification} we are able to
extend the definition of the class of a
smooth projective $G$-variety
in the equivariant Burnside group, so that it is not necessary to
suppose that Assumption \ref{assu.strong} is satisfied.
We can extend, as well, the definition of the class of a
smooth quasiprojective $G$-variety to eliminate the requirement of
Assumption \ref{assu.strong}.

\begin{defi}
\label{defn.main}
Let $X$ be a smooth projective variety over $k$ with
a generically free $G$-action.
The associated class in $\Burn_n(G)$ is obtained by taking
$X'$ to be a smooth projective variety with
$G$-action and equivariant birational morphism to $X$, such that
$X'$ satisfies Assumption \ref{assu.strong}, and setting
\[ [X\actsfromright G]:=[X'\actsfromright G]\in \Burn_n(G). \]
Let $U$ be a smooth quasiprojective variety over $k$ with
a generically free $G$-action.
The associated class is obtained by taking $X$ with $G$-action to be a
smooth projective compactification of $U$ with
$U=X\setminus D$, such that $D=D_1\cup\dots\cup D_\ell$ is a simple normal
crossing divisor on $X$, where each $D_i$ is $G$-invariant, and
$X'$ to be a smooth projective variety with
$G$-action and equivariant birational morphism to $X$, such that
$X'$ satisfies Assumption \ref{assu.strong} and has a simple normal crossing divisor
$D'=D'_1\cup\dots\cup D'_{\ell'}$,
where each $D'_i$ is $G$-invariant,
and $D'$ is the support of the pre-image of $D$ in $X'$;
then
\[ [U\actsfromright G]:=[U'\actsfromright G]\in \Burn_n(G), \]
where $U'$ denotes the complement of $D'$ in $X'$.
\end{defi}

\begin{theo}
\label{thm.maingood}
The classes in Definition \ref{defn.main} are well-defined and give rise
to equivariant birational invariants of
smooth projective, respectively quasiprojective varieties over $k$
with a generically free $G$-action.
\end{theo}

\begin{proof}
For projective varieties this is clear by
Proposition \ref{prop.divisorialification} and
Theorem \ref{thm.ebi}.
For quasiprojective varieties we combine equivariant embedded resolution of
singularieties (for the existence of $D$ as claimed) with
Remark \ref{rem.divisorialification} (for $X'$ and $D'$)
and Proposition \ref{prop.classUebi}.
\end{proof}

\begin{rema}
It is easy to see that the classes $[U\actsfromright G]$ generate $\Burn_n(G)$.
\end{rema}

\section{Specialization}
\label{sec.specialization}
Let $\mathfrak{o}$ be a complete DVR with residue field $k$.
The goal of this section is to produce a specialization map
\[ \Burn_{n,K}(G)\to \Burn_{n,k}(G), \]
where $K$ denotes the field of fractions of $\mathfrak{o}$.
The non-equivariant case was treated in \cite{kontsevichtschinkel}.
A given smooth projective variety over $K$ is extended to
a regular model $\mathcal{X}$ over $\mathfrak{o}$, where the
special fiber is a strict normal crossing divisor.
Then the components of the special fiber, along with their intersections,
are used to define the specialization of $[X]$.
In one variant, the multiplicities of the components play no role.
Another variant involves the multiplicities and is sensitive to the
choice of a uniformizer of $\mathfrak{o}$.

\emph{We fix a choice of uniformizer $\pi\in \mathfrak{o}$.}
Then, from \cite[\S 5]{kontsevichtschinkel}:

\begin{defi}
\label{defn.Burnsidevol}
The \emph{Burnside volume}
\[ \rho_\pi\colon \Burn_{n,K}\to \Burn_{n,k} \]
is defined, for a smooth projective variety $X$ of
dimension $n$ over $K$ and regular model
$\mathcal{X}$ over $\mathfrak{o}$ whose special fiber is a
simple normal crossing divisor $D_1\cup\dots\cup D_\ell$,
with each $D_i$ irreducible, by
\[
\rho_\pi([X]):=
\sum_{\emptyset\ne I\subseteq \cI}(-1)^{|I|-1}[k(\omega_I^{-1}(1))],
\qquad
\cI:=\{1,\dots,\ell\}.
\]
Here,
\[ \omega_I\colon \mathcal{N}^{\circ}_{D_I/\mathcal{X}}\to \G_m \]
denotes the morphism obtained from the trivialization of
\[ \bigotimes_{i\in I} \mathcal{N}^{\otimes d_i}_{D_i/\mathcal{X}}|_{D_I} \]
determined by $\pi$, where
$d_i$ denotes the multiplicity of $D_i$ in the special fiber of $\cX$.
\end{defi}

\begin{exam}
\label{exa.I0star}
Let $E$ be an elliptic curve over $K:=k((t))$ with full $2$-torsion defined over
$K$ and minimal model over $k[[t]]$ of Kodaira type $I_0^*$.
The special fiber consists of four rational curves of multiplicity $1$
and a fifth rational curve of multiplicity $2$.
For $I=\{5\}$, $\omega_I^{-1}(1)$ is a degree $2$ cover of
$\PP^1\setminus \{\text{$4$ points}\}$, while the
contributions from all other $I$ cancel.
The outcome:
$\rho_t([E])$ is the class of an elliptic curve over $k$,
which changes by quadratic twist by $\alpha\in k^\times$ when we
replace $t$ by $\alpha t$.
\end{exam}

\begin{rema}
\label{rem.semistable}
As may be deduced, e.g., from \cite[Prop.\ 2.3.2]{nicaise},
$\omega_I^{-1}(1)$ may as well be obtained from the normalization of
\[ \Spec(\mathfrak{o}[\pi^{1/d_I}])\times_{\Spec(\mathfrak{o})}\mathcal{X}, \]
where $d_I$ denotes the gcd of the integers $d_i$, $i \in I$, as a
$\G_m^{|I|-1}$-torsor over the pre-image of $D^\circ_I$.
In combination with Proposition \ref{prop.classUGrelated}, this quickly
leads to the following observations:
\begin{itemize}
\item If $m$ is a positive integer, then the DVR
$\mathfrak{o}':=\mathfrak{o}[\pi^{1/m}]$ with field of fractions
$K':=K(\pi^{1/m})$ gives rise to a commutative diagram
\[
\xymatrix{
\Burn_{n,K}\ar[r]^{\rho_\pi}\ar[d] & \Burn_{n,k}\ar@{=}[d] \\
\Burn_{n,K'}\ar[r]^{\rho_{\pi^{1/m}}} & \Burn_{n,k}
}
\]
\item To obtain $\rho_\pi([X])$ we may take $m$ such that
a \emph{semistable} model $\mathcal{X}$ over $\mathfrak{o}'$ exists
and apply the specialization map
\[ \Burn_{n,K'}\stackrel{\rho}\longrightarrow \Burn_{n,k} \]
of \cite[\S 3]{kontsevichtschinkel}, where the uniformizing element does not
play a role.
(Such $m$ exists by \cite[\S IV.3]{KKMS}.)
\end{itemize}
\end{rema}

We turn now to the equivariant case.

\begin{defi}
\label{defn.equivBurnsidevol}
We define the \emph{equivariant Burnside volume}
\[ \rho^G_\pi\colon \Burn_{n,K}(G)\to \Burn_{n,k}(G) \]
by, for
$H\in \{H_1,\dots,H_r\}$, smooth projective
$Y$ over $K$ of dimension $d\le n$
with generically free action of $N_G(H)/H$, and $\beta$ a sequence of
$n-d$ elements of $A:=H^\vee$ which generates $A$,
sending
\[ (H,N_G(H)/H\actsfromleft K(Y),\beta) \]
to
\[ \sum_{\emptyset\ne I\subseteq \cI} (-1)^{|I|-1} (H,N_G(H)/H\actsfromleft k(\omega_I^{-1}(1)),\beta),
\qquad \cI:=\{1,\dots,\ell\}. \]
Here we take $\mathcal{Y}$ to be a regular model over $\mathfrak{o}$,
with compatible $N_G(H)/H$-action and special fiber a
simple normal crossing divisor
$D_1\cup\dots\cup D_\ell$, where each $D_i$ is $G$-invariant,
and let
\[ \omega_I\colon \mathcal{N}^{\circ}_{D_I/\mathcal{Y}}\to \G_m \]
denote the morphism obtained from the trivialization of
\[ \bigotimes_{i\in I} \mathcal{N}^{\otimes d_i}_{D_i/\mathcal{Y}}|_{D_I}, \]
determined by $\pi$, where the expression of the special fiber as
$D_1\cup\dots\cup D_\ell$ is taken, so that the components of each $D_i$ have
a common multiplicity $d_i$.
\end{defi}

The verification that Definition \ref{defn.equivBurnsidevol} yields a
well-defined homomorphism is straightforward.
Indeed we recognize that the equivariant Burnside group splits as
a direct sum according to the triviality or nontriviality of the
representation component of triples.
Now we view $Y$ as \emph{quasiprojective} over $\Spec(\mathfrak{o})$ and
apply Definition \ref{defn.main} to $Y\actsfromright N_G(H)/H$, with
compactification $\mathcal{Y}$.
In the definition we disregard Assumption \ref{assu.strong}
and disregard all contributions from
subvarieties with nontrivial generic stabilizer.
We also disregard the first term from \eqref{eqn.classUG}.
Disregarding all contributions from
subvarieties with nontrivial generic stabilizer reduces
Lemma \ref{lem.classUG} to a triviality.
The essential content is contained in Proposition \ref{prop.classUG},
specifically \eqref{eqn.rearrange}.
Since we are disregarding all contributions from
subvarieties with nontrivial generic stabilizer, the
left-hand side trivially vanishes.
The (local) identifications of direct sums of normal bundles that we obtain
from equivariant splittings of short exact sequences of vector bundles
respect the invertible function on punctured normal bundles determined by $\pi$.
So, the cancellations remain valid
when the naive classes are replaced by the classes of the fibers over $1$.
The direct summand of $\Burn_{d,k}(N_G(H)/H)$, determined by the
triviality of the first component of the triples, maps to
$\Burn_{n,k}(G)$ by replacing the first, respectively third
component in each triple by $H$, respectively $\beta$.

\begin{rema}
\label{rem.specializetoorbits}
The class of an irreducible variety with $G$-action can specialize
to an orbit of varieties.
For instance, a non-hyperelliptic curve of genus $3$ with
unramified degree $2$ cover may be presented by the defining equation
\[ Q_1Q_3=Q_2^2, \]
in $\bP^2$, respectively
\[ Q_1=r^2,\qquad Q_2=rs,\qquad Q_3=s^2, \]
in $\bP^4$, where
$Q_1$, $Q_2$, $Q_3\in k[u,v,w]$ are homogeneous of degree $2$,
with $G:=\Z/2\Z$ acting by $(u:v:w:r:s)\mapsto (u:v:w:-r:-s)$
(see \cite{bruin-prym}).
Now we define $X$ over $k((t))$ by replacing $Q_1$, respectively $Q_3$, by
\[ \widetilde{Q}_1:=u^2+tQ_1,\qquad
\text{respectively}\qquad \widetilde{Q}_3:=v^2+tQ_3. \]
For general $Q_1$, $Q_2$, $Q_3$, the same equations define a regular model
$\cX$ where the special fiber consists of two pairs of conics
exchanged by $G$,
and the equivariant Burnside volume of
$[X\actsfromright G]=(\mathrm{triv},G\actsfromleft k(X),\mathrm{triv})$
is a nonzero multiple of
$(\mathrm{triv},G\actsfromleft k(z)\times k(z),\mathrm{triv})$.
\end{rema}

We formulate a version of Assumption \ref{assu.strong} for an
action of $G$ on a regular model $\cX$ of a projective variety $X$ over $K$.
As before, we require all stabilizers of the action to be abelian.
For each $H\in \{H_1,\dots,H_r\}$ and $N_G(H)$-orbit $Y$ of components of
$\cX^H$ where the stabilizer at the generic points of the components is
equal to $H$, we consider two cases:
\begin{itemize}
\item For $Y$ contained in the special fiber we require the composite
\[
\Pic^G(\cX)
\to \rH^1(N_G(H),k(Y)^\times) 
\to
\rH^1(H,k(Y)^\times)^{N_G(H)/H}\to H^\vee
\]
to be surjective.
\item Otherwise, we require the composite
\[
\Pic^G(\cX)
\to \rH^1(N_G(H),K(Y)^\times) 
\to \rH^1(H,K(Y)^\times)^{N_G(H)/H}\to H^\vee
\]
to be surjective.
\end{itemize}

\begin{theo}
\label{theo.ebvXG}
Let $X$ be a smooth projective variety over $K$ with a generically free $G$-action
and regular model $\cX$, projective over $\mathfrak{o}$, to which the
$G$-action extends, whose special fiber is a simple normal crossing divisor
$D_1\cup\dots\cup D_\ell$,
where each $D_i$ is $G$-invariant and
components of each $D_i$ have a common multiplicity $d_i$.
We suppose that the $G$-action on $\cX$ satisfies the above variant of
Assumption \ref{assu.strong}.
Let
\[ \omega_I\colon \mathcal{N}^\circ_{D_I/\cX}\to \G_m \]
denote the morphism obtained from the trivialization of
\[
\bigotimes_{i\in I}\mathcal{N}_{D_i/\cX}|_{D_I},
\]
determined by $\pi$.
Then
\begin{align*}
\rho_\pi^G(&[X\actsfromright G])=\sum_{\emptyset\ne I\subseteq\cI}
(-1)^{|I|-1}\sum_{i=1}^r\\
&\sum_{\substack{\text{$V\subset \mathcal{N}^\circ_{D_I/\cX}$ with}\\ \text{generic stabilizer $H_i$}}}
(H_i,N_G(H_i)/H_i\actsfromleft k(V\cap\omega_I^{-1}(1)),\beta_V(\mathcal{N}^\circ_{D_I/\cX})).
\end{align*}
\end{theo}

\begin{proof}
This follows directly from the definition.
\end{proof}

\begin{rema}
\label{rem.semistableG}
The $G$-equivariant analogues of the statements from
Remark \ref{rem.semistable} are valid.
\end{rema}

\begin{coro}
\label{cor.ebvXG}
Let $X$ and $X'$ be smooth projective varieties over $K$
with generically free $G$-actions, admitting regular models
$\cX$, respectively $\cX'$, smooth and projective over $\mathfrak{o}$, to
which the $G$-action extends.
If $X$ and $X'$ are $G$-equivariantly birational over $K$,
then the special fibers of $\cX$ and $\cX'$ are
$G$-equivariantly birational over $k$.
\end{coro}

The following generalizes the notion of
$B$-rational singularities of \cite{kontsevichtschinkel}.

\begin{defi}
\label{def.BGrat}
Let $X_0$ be a singular projective variety over $k$ with a
generically free $G$-action.
We say that $X_0$, respectively a pair $(\cX,X_0)$
has \emph{$BG$-rational singularities} if for every
projective model $\cX$ over $\mathfrak{o}$, respectively a given
projective model $\cX$,
with $G$-action, smooth generic fiber $X$, and special fiber
$G$-equivariantly isomorphic to $X_0$, we have
\[ \rho^G_\pi([X\actsfromright G])=[X_0\actsfromright G]. \]
\end{defi}

\begin{exam}
\label{exa.ODPisBGrat}
If the only singularity of $X_0$ is an
orbit of isolated ordinary double point singularities on which
$G$ acts simply transitively,
then $X_0$ is $BG$-rational.
Indeed, any projective model $\cX$
is resolved by a sequence of
blow-ups of orbits of points, and it is straightforward to verify using
Theorem \ref{theo.ebvXG} that $[X\actsfromright G]$ specializes to
$[X_0\actsfromright G]$.
\end{exam}

\section{Equivariant Burnside groups and orbifolds}
\label{sect:orbi}
In this section, we define a natural homomorphism 
from the equivariant Burnside group to the Burnside group of orbifolds:
\[
\kappa^G: {\Burn}_n(G)                  \to         {\oBurn}_n,
\]
a group we introduced in \cite{Bbar}, 
as a quotient, by explicit relations, of the $\Z$-module with generators 
\[
([X], [\beta]),
\]
where 
\begin{itemize}
\item $[X] \in \Burn_d$, $d\le n$, and 
\item $[\beta] \in \overline{\mathcal B}_{n-d}$,  a certain 
invariant of representations of finite abelian groups \cite[Definition 3.1]{Bbar}.
\end{itemize}
Then the class in ${\oBurn}_n$ of an $n$-dimensional orbifold $\cX$ is defined  
as follows (see \cite[Section 4]{Bbar} for terminology and precise definitions): 
\begin{itemize}
\item After divisorialification we may assume that  
$\cX$ is {\em divisorial} (with respect to some finite collection of line bundles);
\item Express 
\[
\cX=\coprod_H \cX_H, 
\]
where $\cX_H$ are strata characterized by the isomorphism type of 
their geometric stabilizer group $H$ (an abelian group); 
\item Put
\begin{equation}
\label{eqn.classcX}
[\cX]:=\sum_{H} ([X_H],[N_{X,H}])\in {\oBurn}_n,
\end{equation} 
where $X_H$ is the coarse moduli space of $\cX_H$, 
and the representation $[N_{X,H}]$ is extracted from
the normal bundle $N_{X,H}:=\mathcal{N}_{\cX_H/\cX}$.
\end{itemize}
This class is a well-defined invariant of the birational type of $\cX$. 

\begin{exam}
\label{exa.linearquotient}
Let an abelian group $H$ act trivially on a
smooth projective $d$-dimensional variety $Y$ and diagonally
by characters $a_1$, $\dots$, $a_{n-d}$ generating
$A:=H^\vee$ on $\A^{n-d}$.
Then
\begin{equation}
\label{eqn.YAH}
[Y\times \A^{n-d}/H]=\sum_{I\subseteq \{1,\dots,n-d\}}
(-1)^{|I|}([Y],(\bar a_1,\dots,\bar a_{n-d})),
\end{equation}
where $\bar a_i$ denotes the character of $H_I:=\bigcap_{i\in I}\ker(a_i)$ given
by the class of $a_i$ in $A_I:=A/\langle a_i\rangle_{i\in I}$.
We obtain \eqref{eqn.YAH} from \eqref{eqn.classcX}
by repeatedly applying modified scissors relations
\cite[\S 2]{Bbar} to subvarieties defined by the vanishing of
a coordinate of $\A^{n-d}$.
\end{exam}

\begin{rema}
\label{rem.linearquotientzero}
Suppose that $a_i=0$ for some $i$ in Example \ref{exa.linearquotient}.
Then $[Y\times \A^{n-d}/H]$ is a product with $\A^1$.
Hence \eqref{eqn.YAH} vanishes by the triviality in $\Burn_n$
of any product with $\A^1$
(cf.\ the proof of \cite[Prop.\ 2.2]{Bbar}).
Another way to obtain the vanishing is to notice that in the sum \eqref{eqn.YAH}
the terms indexed by $I$ and $I\cup\{i\}$ cancel, for $i\notin I$.
\end{rema}

In Example \ref{exa.linearquotient} the orbifold in question is the
total space of a direct sum of line bundles over $Y\times BH$.
More generally, we may consider a smooth projective
Deligne-Mumford stack $\cX$ with line bundles
$L_1$, $\dots$, $L_r$, such that the
total space of $L_1\oplus\dots\oplus L_r$ is an orbifold.
It will be convenient to require this orbifold to be divisorial.
We recall a well-known fact
\cite[Prop.\ 2.1]{olssonboundedness}:
\emph{Every smooth separated finite-type Deligne-Mumford stack over a field
is a gerbe over an orbifold}.
A smooth projective (respectively, quasiprojective)
Deligne-Mumford stack $\cX$ is a gerbe, then,
over a projective (respectively, quasiprojective) orbifold $\overline{\cX}$.

\begin{lemm}
\label{lem.star}
Let $\cX$ be a a smooth separated irreducible Deligne-Mumford stack of
finite type over $k$,
gerbe over the orbifold $\overline{\cX}$, with finite abelian group
$H$ as the isomorphism type of the geometric generic point of $\cX$, and let
$L_1$, $\dots$, $L_r$ be line bundles on $\cX$.
The following are equivalent.
\begin{itemize}
\item[(i)] The orbifold $\overline{\cX}$ is divisorial, and the
representation of $H$ at the geometric generic point of $\cX$ determined
by $L_1\oplus\dots\oplus L_r$ is faithful.
\item[(ii)] The total space of
$L_1\oplus\dots\oplus L_r$
is an orbifold and is divisorial.
\item[(iii)] The total space of
$L_1\oplus\dots\oplus L_r$
is an orbifold and is divisorial with respect to the pullbacks of
$L_1$, $\dots$, $L_r$ and some finite collection of line bundles
on $\overline{\cX}$.
\end{itemize}
\end{lemm}

\begin{proof}
The kernel of the representation in (i) is the geometric generic stabilizer
of the total space of
$L_1\oplus\dots\oplus L_r$.
This lets us rephrase (ii) and (iii), replacing the requirement for the
total space of $L_1\oplus\dots\oplus L_r$ to be an orbifold by the
faithful representation requirement from (i).
For an orbifold to be divisorial with respect to a
finite collection of line bundles, we recall, means that the corresponding
morphism to a product of copies of $B\G_m$ is representable;
we allow ourselves to
apply this terminology to arbitrary Deligne-Mumford stacks.
This way, we can formulate variants of (ii) and (iii), let us say
$\mathrm{(ii')}$ and $\mathrm{(iii')}$, in which instead of requiring
$L_1\oplus\dots\oplus L_r$ to be divisorial we require $\cX$ to be divisorial.
Since $\cX$ sits as the zero-section in $L_1\oplus\dots\oplus L_r$, we have
$\mathrm{(ii)} \Rightarrow \mathrm{(ii')}$ and
$\mathrm{(iii)} \Rightarrow \mathrm{(iii')}$.
As well, $L_1\oplus\dots\oplus L_r\to \cX$ is representable, so
$\mathrm{(ii')} \Rightarrow \mathrm{(ii)}$ and
$\mathrm{(iii')} \Rightarrow \mathrm{(iii)}$.
Trivially, $\mathrm{(iii)} \Rightarrow \mathrm{(ii)}$.

We establish $\mathrm{(i)} \Rightarrow \mathrm{(iii')}$ by the
observation that the stabilizer at a geometric point of $\cX$ contains $H$,
such that the quotient by $H$ is the stabilizer at the corresponding
geometric point of $\overline{\cX}$.
We are done, then, if we can establish $\mathrm{(ii')} \Rightarrow \mathrm{(i)}$.
A line bundle on $\cX$ comes from $\overline{\cX}$ if and only if induces
the trivial representation of $H$, according to the criterion of
\cite[Thm.\ 10.3]{alper} (in a form adapted to a relative notion of
coarse moduli space, cf.\ \cite[\S 2.3]{oesinghaus}).
Since we have a faithful representation as in (i), we get a surjective
homomorphism of character groups $\Z^r\to H^\vee$.
So any line bundle on $\cX$ will, after adjustment by a suitable
tensor combination of $L_1$, $\dots$, $L_r$, descend to $\overline{\cX}$.
Carrying this out for a collection of line bundles, with respect to which
$\cX$ is divisorial, we get vector bundles, with respect to which
$\overline{\cX}$ is divisorial
\end{proof}

Given $\cX$, $L_1$, $\dots$, $L_r$ satisfying the equivalent conditions
of Lemma \ref{lem.star}, with $\cX$ quasiprojective of dimension $n-r$,
we define, following the notation of \cite[Thm.\ 4.1]{Bbar},
\[ [\cX,(L_1,\dots,L_r)]:=\sum_P ([X_P],[N_{X,P}\oplus L_1\oplus\dots\oplus L_r])
\in \oBurn_n, \]
sum over geometric stabilizer groups $P$ of $\cX$.

\begin{lemm}
\label{lem.orbifolddivisors}
Let $\cX$, $L_1$, $\dots$, $L_r$ satisfy the equivalent conditions of
Lemma \ref{lem.star}, with $\cX$ quasiprojective of dimension $n-r$,
and let $D=D_1\cup\dots\cup D_\ell$ be a
simple normal crossing divisor on $\cX$.
Then, with $\cU:=\cX\setminus D$ and $\cI:=\{1,\dots,\ell\}$ we have
\begin{align}
\begin{split}
\label{eqn.orbifolddivisors}
[&\cU,(L_1|_\cU,\dots,L_r|_\cU)]
=[\cX,(L_1,\dots,L_r)] \\
&\qquad+\sum_{\emptyset\ne I\subseteq \cI} (-1)^{|I|}
[D_I,(L_1|_{D_I},\dots,L_r|_{D_I},\dots,\mathcal{N}_{D_i/\cX}|_{D_I},\dots)]
\end{split}
\end{align}
in $\oBurn_n$, where the last term includes the line bundles
$\mathcal{N}_{D_i/\cX}|_{D_I}$ for all $i\in I$.
\end{lemm}

\begin{proof}
We first remark that \eqref{eqn.orbifolddivisors},
when $\cX$ is a projective variety,
is essentially the formula in $\Burn_n$ for the class of
the complement of a strict normal crossing divisor \cite[(1.1)]{Bbar}.
It is an easy exercise to verify that the formula
is valid when $\cX$ is a quasiprojective variety.

To treat the case of a Deligne-Mumford stack,
we consider a geometric stabilizer group $P$ of $\cX$.
It is a straightforward fact that $\cX_P$
has normal crossing with $D_1$, $\dots$, $D_\ell$.
We obtain the lemma from an equality for each $P$.
When (a connected component of) $\cX_P$ is contained in some $D_i$,
there is no contribution to the left-hand side, while the right-hand side,
viewed as a single sum over all $I\subseteq \cI$, has pairs of terms that cancel,
indexed by $I$ and $I\cup \{i\}$ for $i\notin I$.
Otherwise, each $L_i$ is generically trivial on $D_I$, for all $I$,
and the equality may be rewritten as
\begin{align}
\begin{split}
\label{eqn.UGNUG}
([&U_P],[N_{U,P}\oplus L_1\oplus\dots\oplus L_r])=
([X_P],[N_{X,P}\oplus L_1\oplus\dots\oplus L_r])\\&\qquad+
\sum_{\emptyset\ne I\subseteq \cI} (-1)^{|I|}
([(D_I)_P\times \PP^{|I|}],[N_{D_I,P}\oplus L_1\oplus\dots\oplus L_r]),
\end{split}
\end{align}
where we write $(D_I)_P$ for the
coarse moduli space of $D_I\cap \cX_P$ and neglect to include in the
notation that the line bundles are restricted to $U$, respectively, to
$D_I$.
The same representation is extracted from the vector bundles in
all of the terms in \eqref{eqn.UGNUG}.
So, the equality follows from the case of a
quasiprojective variety mentioned above.
\end{proof}

\begin{lemm}
\label{lem.coarse}
Let $\cX$ be an $n$-dimensional quasiprojective orbifold over $k$ that is
divisorial with respect to some collection of line bundles, and let
$D=D_1\cup\dots\cup D_\ell$ be a simple normal crossing divisor with complement
$\cU$.
Then, with $\cI:=\{1,\dots,\ell\}$,
the coarse moduli spaces $X$ of $\cX$, etc., satisfy
\[ [U]=[X]+\sum_{\emptyset\ne I\subseteq \cI}(-1)^{|I|}[D_I\times \PP^{|I|}] \]
in $\Burn_n$.
\end{lemm}

\begin{proof}
The equality holds if and only if it holds after blow-up of any
smooth substack of $\cX$ that has normal crossings with $D_1$, $\dots$, $D_\ell$
(where the exceptional divisor gets added as $D_{\ell+1}$ in case of center of
blow-up contained in $D$), cf.\ the proof of \cite[Thm.\ 4]{kontsevichtschinkel}.
We conclude by applying the destackification algorithm \cite{bergh},
which leads to a situation where $X$ is smooth with simple
normal crossing divisor determined by the $D_i$.
Then the desired formula is \cite[(1.1)]{Bbar} which, as mentioned in the proof of
Lemma \ref{lem.orbifolddivisors}, is also valid in the setting of
quasiprojective $X$.
\end{proof}

\begin{defi}
\label{defn.toBurnbar}
Let $K_0$ be a finitely generated field over $k$ and
$K/K_0$ a Galois algebra for the group $N_G(H)/H$.
Put
\[
\kappa^G\colon (H,N_G(H)/H\actsfromleft K,\beta) \mapsto [Y\times \A^{n-d}/H],
\]
where $H$ acts trivially on $Y$, any smooth projective variety with
function field $K_0$, and via the representation $\beta$ on $\A^{n-d}$.
\end{defi}

\begin{prop}
\label{prop.toBurnbarwelldefined}
We have a well-defined group homomorphism
\[
\kappa^G\colon {\Burn}_n(G) \to {\oBurn}_n.
\]
\end{prop}

\begin{proof}
Changing the triple by a conjugation relation does not
change the isomorphism type of the quotient stack
$[Y\times \A^{n-d}/H]$, hence these relations are respected by the
map in Definition \ref{defn.toBurnbar}.
To show that the map respects \textbf{(B1)}, i.e.,
\[
[Y\times \A^{n-d}/H]=0\in \oBurn_n,
\] 
when $H$ acts on
$\A^{n-d}$ with weights $(a_1,\dots,a_{n-d})$, and
$a_1+a_2=0$, we let $\cY:=[Y\times\A^{n-d-2}/H]$ 
with action by weights $a_3$, $\dots$, $a_{n-d}$, and $L$ the line bundle given
by the weight $a_1$.
Using Example \ref{exa.linearquotient}, we verify
\[
[Y\times \A^{n-d}/H]=[\cY,(L,L^\vee)]
-[Y\times \PP^2\times \A^{n-d-2}/\ker(a_1)]\in \oBurn_n.
\]
The right-hand side vanishes, as we see by applying
the final relation of \cite[Defn.\ 3.1]{Bbar} with $a_2=0$ and $j=2$.

Let us rewrite \textbf{(B2)} as
\[
(H,N_G(H)/H \actsfromleft K, \beta)-\Theta_2=
\Theta_1.
\]
To see that the map respects \textbf{(B2)}, we first verify that the
image of $-\Theta_2$ in $\oBurn_n$ is
\begin{equation}
\label{eqn.imagemTheta2}
[Y\times (\A^1\setminus\{0\})\times \A^{n-d-1}/H],
\end{equation}
where the action of $H$ is by $a_1-a_2$ on
$\A^1\setminus\{0\}$ and by $a_2$, $\dots$, $a_{n-d}$ on $\A^{n-d-1}$.
There are two cases:
\begin{itemize}
\item
If $a_i\in \langle a_1-a_2\rangle$ for some $i$; then after
replacing the corresponding coordinate of
$\A^{n-d-1}$ by its product with a suitable power of the
coordinate of $\A^1\setminus\{0\}$, we have a trivial $H$-action on this
coordinate.
As in Remark \ref{rem.linearquotientzero},
\eqref{eqn.imagemTheta2} is $0$, while $\Theta_2=0$ by definition.
\item 
Otherwise, $\Theta_2$ is the class of the triple
given in Definition \ref{defn.eBg}, and we have a formula for the
image in $\oBurn_n$ of $-\Theta_2$ by Example \ref{exa.linearquotient}.
Reasoning as in Example \ref{exa.linearquotient}, we obtain a formula for
\eqref{eqn.imagemTheta2}.
Comparing, we find that the two expressions are equal.
\end{itemize}

Thus it suffices to show that
\[ [Y\times \A^{n-d}/H]+[Y\times (\A^1\setminus\{0\})\times \A^{n-d-1}/H]=
[Y\times \A^{n-d}/H]+[Y\times \A^{n-d}/H] \]
holds in $\oBurn_n$,
where the first term on the left corresponds to the
given triple, and the terms on the right correspond to the triples in
$\Theta_1$.
Equivalently, by \cite[Thm.\ 4.1]{Bbar},
\begin{align*}
[Y\times (B\ell_0\A^2)\times\A^{n-d-2}/H]&+
[Y\times (\A^1\setminus\{0\})\times \A^{n-d-1}/H]\\
&\qquad =[Y\times \A^{n-d}/H]+[Y\times \A^{n-d}/H].
\end{align*}
The stack $[Y\times (\A^1\setminus\{0\})\times \A^{n-d-1}/H]$ as well
as each of the stacks $[Y\times \A^{n-d}/H]$ on the right-hand side
may be viewed as complements of smooth divisors in
$[Y\times (B\ell_0\A^2)\times\A^{n-d-2}/H]$.
We get the desired equality by Lemma \ref{lem.orbifolddivisors}.
\end{proof}

In order to show that $\kappa^G$ maps the class of a variety with
group action to the class of the associated quotient stack
(Proposition \ref{prop.toBurnbar}), we need the following computation,
where we employ the notation
\[ [X\actsfromright G, (L_1,\dots,L_r)]^{\mathrm{naive}}:=[L_1\oplus\dots\oplus L_r\actsfromright G]^{\mathrm{naive}}\in \Burn_n(G) \]
for a smooth quasiprojective variety $X$ with $G$-action and
$G$-linearized line bundles $L_1$, $\dots$, $L_r$, with $X$ of
dimension $n-r$, such that the $G$-action on
the total space of $L_1\oplus\dots\oplus L_r$ is generically free.

\begin{lemm}
\label{lem.toBurnbar}
Let $X$ be a smooth quasiprojective variety with
$G$-action, and let
$L_1$, $\dots$, $L_r$ be $G$-linearized line bundles on $X$, such that
the $G$-action on $L_1\oplus\dots\oplus L_r$ is generically free.
We suppose that $G$ acts transitively on the set of components of $X$,
the dimension of $X$ is $n-r$, and for some component of $X$, every point
has stabilizer $H\subseteq G$.
Then $\kappa^G([X\actsfromright G, (L_1,\dots,L_r)]^{\mathrm{naive}})$
is the class in $\oBurn_n$ of the field $k(X)^G$ and the
representation of $H$ determined by $L_1\oplus\dots\oplus L_r$.
\end{lemm}

\begin{proof}
Let $A:=H^\vee$, with $a_1$, $\dots$, $a_r\in A$ determined by
$L_1$, $\dots$, $L_r$.
With $\cI:=\{1,\dots,r\}$ and, for $I\subseteq \cI$,
\[ H_I:=\bigcap_{i\in I}\ker(a_i)\qquad\text{and}\qquad A_I:=A/\langle a_i\rangle_{i\in I}, \]
we have an expression of $[X\actsfromright G,(L_1,\dots,L_r)]^{\mathrm{naive}}$
as a sum of triples over $I\subseteq \cI$, such that
for all $j\in \cI\setminus I$ we have $a_j\notin \langle a_i\rangle_{i\in I}$.
The triple for a given $I$ consists of the group $H_I$, a Galois algebra over
$k(X)^G(t_1,\dots,t_{|I|})$, and the classes of $a_j$ in $A_I$, for
$j\in \cI\setminus I$.
By Example \ref{exa.linearquotient}, upon application of $\kappa^G$ such a triple
gives rise to a sum indexed by subsets of $\cI$, disjoint from $I$:
\begin{equation}
\label{eqn.sumprimesum}
\kappa^G([X\actsfromright G, (L_1,\dots,L_r)]^{\mathrm{naive}})=
{\sum_{I\subseteq \cI}}'\sum_{\substack{J\subseteq \cI\\ I\cap J\ne \emptyset}}
(-1)^{|J|}([k(X)^G],(\bar a_1,\dots,\bar a_r)),
\end{equation}
where $\sum'$ denotes the sum over $I$ satisfying
$a_j\notin \langle a_i\rangle_{i\in I}$ for all $j\in \cI\setminus I$,
and $\bar a_i$ denotes the class of $a_i$ in $A_{I\cup J}$.
By Remark \ref{rem.linearquotientzero}, if we consider some $I$ such that
$a_j\in \langle a_i\rangle_{i\in I}$ for some $j\in \cI\setminus I$ and
evaluate the inner sum of the right-hand side of
\eqref{eqn.sumprimesum} for this $I$, we get $0$.
So
\[
\kappa^G([X\actsfromright G, (L_1,\dots,L_r)]^{\mathrm{naive}})=
\sum_{\substack{I,J\subseteq I\\ I\cap J\ne \emptyset}}
(-1)^{|J|}([k(X)^G],(\bar a_1,\dots,\bar a_r)).
\]
Due to the sign, only the term with $I=J=\emptyset$ survives.
\end{proof}

\begin{prop}
\label{prop.toBurnbar}
For a quasiprojective variety $X$ of dimension $n$ with generically free $G$-action
and stack quotient $\cX:=[X/G]$ we have
\[
\kappa^G([X \actsfromright G])=[\cX]
\]
in $\oBurn_n$.
\end{prop}

\begin{proof}
Definition \ref{defn.classUG} and Lemma \ref{lem.orbifolddivisors} allow us
to reduce the proposition to the claim, that for a
smooth projective variety $X$ with $G$-action and $G$-linearized line bundles
$L_1$, $\dots$, $L_r$, such that the $G$-action on
$L_1\oplus \dots\oplus L_r$ is generically free and satisfies
Assumption \ref{assu.strong}, with $X$ of dimension $n-r$, we have
\begin{equation}
\label{eqn.reducetoX}
\kappa^G([X\actsfromright G,(L_1,\dots,L_r)]^{\mathrm{naive}})=
[\cX,(L_1,\dots,L_r)],
\end{equation}
where the line bundles on $\cX$ are those determined by the
$G$-linearized line bundles $L_1$, $\dots$, $L_r$.

We prove \eqref{eqn.reducetoX} by induction on $\dim(X)$.
The case $\dim(X)=0$ follows from Lemma \ref{lem.toBurnbar}.
So we assume $\dim(X)>0$.

The class in $\Burn_n(G)$ on the left-hand side of \eqref{eqn.reducetoX} is
a birational invariant, by Lemma \ref{lem.Uebi}.
If we blow up $X$ along a smooth $G$-invariant subvariety, then
$\cX$ gets blown up along the corresponding smooth substack.
We verify that the right-hand side of \eqref{eqn.reducetoX} remains unchanged,
by observing that the equality \cite[(4.1)]{Bbar} is valid when
$L_1\oplus\dots\oplus L_r$ is inserted on both sides.
So, by divisorialification
as in Proposition \ref{prop.divisorialification},
we may reduce further to
the case that $X$ possesses a simple normal crossing divisor
$D=D_1\cup\dots\cup D_\ell$, where each $D_i$ is $G$-invariant, such that
$U:=X\setminus D$ has constant stabilizers, as in Lemma \ref{lem.toBurnbar}.

With a minor adaptation to the proof of Lemma \ref{lem.classUG},
taking $W$ in $L_1|_{D_M}\oplus\dots\oplus L_r|_{D_M}$ instead of in $D_M$,
we obtain the identity
\begin{align*}
[&X \actsfromright G, (L_1,\dots,L_r)]^{\mathrm{naive}} =
[ U \actsfromright G, (L_1|_U,\dots,L_r|_U)]^{\mathrm{naive}} \\
&\,\,\,\,\,+\sum_{\emptyset\ne I\subseteq \cI} (-1)^{|I|}
[\mathcal{N}^\circ_{D_I/X} \actsfromright G, (L_1|_{D_I},\dots,L_r|_{D_I})]^{\mathrm{naive}} \\
&\!\!\!\!\!\!\!-\sum_{\emptyset\ne I\subseteq \cI} (-1)^{|I|}
[D_I \actsfromright G,(L_1|_{D_I},\dots,L_r|_{D_I},\dots,\mathcal{N}_{D_i/X}|_{D_I},\dots)]^{\mathrm{naive}},
\end{align*}
where the last term includes the line bundles
$\mathcal{N}_{D_i/X}|_{D_I}$ for all $i\in I$.
After applying $\kappa^G$,
Lemma \ref{lem.toBurnbar} is applicable to the first two terms on the right,
then Lemma \ref{lem.coarse} lets us identify their contribution as the
class in $\Burn_{n-r}$ of the quotient variety of $U$ by $G$, paired with the
class in $\ocB_r$ of the
representation of the generic stabilizer determined by $L_1$, $\dots$, $L_r$.
The induction hypothesis is applicable to the final term on the right,
and we conclude by Lemma \ref{lem.orbifolddivisors}.
\end{proof}

\section{Comparisons}
\label{sect:comp}

In this section, $G$ is a finite \emph{abelian} group.
We will compare birational invariants
for actions of such $G$ introduced in 
\cite{kontsevichpestuntschinkel} with the invariant in Definition \ref{defn.classXG}, taking values in  
the equivariant Burnside group 
\[
\Burn_n(G)
\] 
from Definition \ref{defn.eBg}. 

Let $A:=G^\vee$ be the character group of $G$.
We recall from \cite{kontsevichpestuntschinkel}:
\[
\cB_n(G)
\]
is the $\Z$-module 
generated by equivalence classes of faithful
$n$-dimensional linear representations of $G$ over $k$, 
i.e., by symbols 
\begin{equation}
\label{eqn.BnGsymbol}
\beta:=[a_1,\dots, a_n], \quad \quad a_i\in A, 
\end{equation}
consisting of $n$-tuples of characters of $G$, up to order, generating $A$;
these are subject to relations, for all $2\le j\le n$:
\begin{align}
\begin{split}
\label{eqn.BnGrel}
[&a_1,\dots,a_n]\\
&\qquad=\sum_{\substack{1\le i\le j\\ a_i\ne a_{i'}\,\forall\,i'<i}}    
[a_1-a_i, \dots, a_i,  \dots, a_j-a_i, a_{j+1}, \dots, a_n].
\end{split}
\end{align}

We first explain how to obtain $\cB_n(G)$ as a quotient of $\Burn_n(G)$.
There is a quotient group $\Burn_n^G(G)$ of $\Burn_n(G)$,
consisting of triples with first entry $G$:
\[
(G,\mathrm{triv}\actsfromleft K,\beta).
\]
The homomorphism
\[ \Burn_n(G)\to \Burn_n^G(G) \]
to the quotient group annihilates all triples whose first entry is a
proper subgroup of $G$.
The next result reveals a further quotient group, isomorphic to $\cB_n(G)$.

\begin{prop}
\label{prop.toBnG}
The map sending the class of a triple
\[ 
(G,\mathrm{triv}\actsfromleft K,\beta) \in \Burn_n^G(G),
\]
with
$\beta=(a_1,\dots,a_{n-d})$, and $d$ the transcendence degree of
$K$ over $k$, to
\[ [k':k][a_1,\dots,a_{n-d},0,\dots,0] \in  \cB_n(G),
\]
where $k'$ denotes the algebraic closure of $k$ in $K$, extends to a
surjective homomorphism
\[ \Burn_n^G(G)\to \cB_n(G). \]
For $n\ge 2$, the group $\cB_n(G)$ may be presented as a $\Z$-module
by the generators \eqref{eqn.BnGsymbol} and
just the relations \eqref{eqn.BnGrel} with $j=2$.
\end{prop}

\begin{proof}
The quantity $[k':k]$ is unchanged when $K$ is replaced by $K(t)$.
So, to verify that the map is a homomorphism, it suffices to verify relations
analogous to \textbf{(B1)} and \textbf{(B2)} in
$\cB_n(G)$.
For $\textbf{(B1)}$, we apply \eqref{eqn.BnGrel} with $j=2$ and $a_2=0$.
The general $j=2$ case of \eqref{eqn.BnGrel} takes care of
\textbf{(B2)}.
It is clear that the homomorphism is surjective.
For the final claim, we just have to check that in the abelian group defined by
generators \eqref{eqn.BnGsymbol} and
relations \eqref{eqn.BnGrel} with $j=2$, the relations
\eqref{eqn.BnGrel} with $j>2$ follow as a consequence.
If any of $a_1$, $\dots$, $a_j$ is $0$, then the relation is a consequence
of the relation analogous to \textbf{(B1)}, which we have already treated.
Otherwise, the relation is analogous to 
Proposition \ref{prop.eBg} (ii) and thus follows from relations
analogous to \textbf{(B1)} and \textbf{(B2)}.
\end{proof}

A refined version of $\cB_n(G)$ keeps information about
the stable birational type of $K$.
This refined group
\[ \cB_n(G,k) \]
is defined by associating to a function field $K$
of transcendence degree $d$ over $k$:
\begin{itemize}
\item The isomorphism type of $K(t_1,\dots,t_{n-1-d})$, encoding the
stable birational type of a model of $K$, when the formation
of products with $\bP^1$ is restricted by the requirement of dimension $\le n-1$.
\item An element of $\cB_{m+1}(G)$, where $m$ is the
largest integer such that $K(t_1,\dots,t_{n-1-d})$ is
isomorphic over $k$ to $L(u_1,\dots,u_m)$ for some field $L$,
finitely generated over $k$.
\end{itemize}
We refer the reader to \cite{kontsevichpestuntschinkel} for a more
detailed description and the definition of the class
$\beta_k(X)\in \cB_n(G,k)$ of a smooth projective variety $X$ with faithful $G$-action.

\begin{prop}
\label{prop.toBnGk}
The surjective homomorphism from Proposition \ref{prop.toBnG} factors through
$\cB_n(G,k)$,
and the homomorphism
\begin{equation}
\label{eqn.toBnGk}
\Burn_n^G(G)\to \cB_n(G,k)
\end{equation}
is surjective.
If $X$ is a smooth projective variety with a generically free
action of $G$, then the image of
$[X\actsfromright G]$ in $\Burn_n^G(G)$ maps to
$\beta_k(X)$ under the map \eqref{eqn.toBnGk}.
\end{prop}

\begin{proof}
The proof is analogous to that of Proposition \ref{prop.toBnG}.
\end{proof}

\bibliographystyle{plain}
\bibliography{BnG}

\end{document}